\documentclass[12pt]{amsart}
\usepackage{latexsym,amsmath,amssymb,graphicx}
\usepackage{tikz}\usetikzlibrary{matrix,arrows,calc}
\graphicspath{{.}{./figs/}}

\newcommand{\PropIdentMin}{3.5}
\newcommand{\LemOrthInv}{5.3}
\newcommand{\PropMoveRefl}{6.2}
\newcommand{\LemMinEll}{6.4}
\newcommand{\PropEllRefl}{6.6}

\theoremstyle{plain}
\newtheorem{thm}{Theorem}[section]
\newtheorem{main}{Theorem}
\newtheorem{lem}[thm]{Lemma}
\newtheorem{prop}[thm]{Proposition}
\newtheorem{cor}[thm]{Corollary}

\theoremstyle{definition}
\newtheorem{defn}[thm]{Definition}
\newtheorem{rem}[thm]{Remark}
\newtheorem{exmp}[thm]{Example}

\newcommand{\N}{\mathbb{N}}
\newcommand{\Z}{\mathbb{Z}}

\newcommand{\R}{\mathbb{R}}

\newcommand{\fix}{\textsc{Fix}}
\newcommand{\ms}{\textsc{Min}}
\newcommand{\dir}{\textsc{Dir}}
\newcommand{\mov}{\textsc{Mov}}
\newcommand{\lin}{\textsc{Lin}}
\newcommand{\aff}{\textsc{Aff}}
\newcommand{\spn}{\textsc{Span}}
\newcommand{\isom}{\textsc{Isom}}
\newcommand{\cay}{\textsc{Cay}}

\newcommand{\cox}{\textsc{Cox}}
\newcommand{\art}{\textsc{Art}}
\newcommand{\dart}{\textsc{Art${}^*$}}

\newcommand{\inv}{\textrm{inv}}

\newcommand{\wt}{\widetilde}
\newcommand{\onto}{\twoheadrightarrow}

\newcommand{\lr}{\ell_R}\newcommand{\ls}{\ell_S}\newcommand{\lre}{\ell_{R_E}}

\newcommand{\join}{\vee}
\newcommand{\meet}{\wedge}
\newcommand{\bigjoin}{\bigvee}
\newcommand{\bigmeet}{\bigwedge}

\newcommand{\bc}{\begin{center}}\newcommand{\ec}{\end{center}}
\newcommand{\bt}{\begin{tabular}}\newcommand{\et}{\end{tabular}}

\newcommand{\drawEdge}[2]{\draw[-,thick] #1--#2;}
\newcommand{\drawDashEdge}[2]{\draw[-,dashed,thick] #1--#2;}
\newcommand{\drawArrow}[3]{
  \draw[-,thick] ($#1!.45!#2$)--($#1!.55!#2+#3$);
  \draw[-,thick] ($#1!.45!#2$)--($#1!.55!#2-#3$);
}
\newcommand{\drawTripleEdge}[3]{
  \draw[-,thick] #1--#2; 
  \draw[-,thick,yshift=.5mm] #1--#2; 
  \draw[-,thick,yshift=-.5mm] #1--#2;
  \drawArrow{#1}{#2}{#3}
}
\newcommand{\drawDoubleEdge}[3]{
  \draw[-,thick,yshift=.3mm] #1--#2; 
  \draw[-,thick,yshift=-.3mm] #1--#2;
  \drawArrow{#1}{#2}{#3}
}
\newcommand{\drawDashDoubleEdge}[3]{
  \draw[-,dashed,thick,yshift=.3mm] #1--#2; 
  \draw[-,dashed,thick,yshift=-.3mm] #1--#2;
  \drawArrow{#1}{#2}{#3}
}
\newcommand{\drawInfEdge}[2]{
  \drawDashEdge{#1}{#2} 
  \node[anchor=south] at ($#1!.5!#2$) {$\infty$};
}
\newcommand{\drawDotEdge}[2]{
  \fill ($#1!.3!#2$) circle (.4mm);
  \fill ($#1!.5!#2$) circle (.4mm);
  \fill ($#1!.7!#2$) circle (.4mm);
}
\newcommand{\drawRegDot}[1]{\fill #1 circle (.7mm);\fill[color=black!80] #1 circle (.5mm);}
\newcommand{\drawSpeDot}[1]{\fill #1 circle (1mm);\fill[color=white] #1 circle (.7mm);}
\newcommand{\drawESpDot}[1]{\fill #1 circle (1mm);\fill[color=red!70] #1 circle (.7mm);}

\begin{document}
%%%%%%%%%%%%%%%%

\title[Dual Artin groups of euclidean type]{Dual euclidean Artin
  groups and the failure of the lattice property}

\author{Jon McCammond}
\address{Dept. of Math., University of California, Santa Barbara, CA
  93106} 
\thanks{Support by the National Science Foundation is gratefully
  acknowledged}
\date{\today}
\email{jon.mccammond@math.ucsb.edu}

\begin{abstract}
  The irreducible euclidean Coxeter groups that naturally act
  geometrically on euclidean space are classified by the well-known
  extended Dynkin diagrams and these diagrams also encode the modified
  presentations that define the irreducible euclidean Artin groups.
  These Artin groups have remained mysterious with some exceptions
  until very recently.  Craig Squier clarified the structure of the
  three examples with three generators more than twenty years ago and
  Fran\c{c}ois Digne more recently proved that two of the infinite
  families can be understood by constructing a dual presentation for
  each of these groups and showing that it forms an infinite-type
  Garside structure.  In this article I establish that none of the
  remaining dual presentations for irreducible euclidean Artin groups
  corrspond to Garside structures because their factorization posets
  fail to be lattices.  These are the first known examples of dual
  Artin presentations that fail to form Garside structures.
  Nevertheless, the results presented here about the cause of this
  failure form the foundation for a subsequent article in which the
  structure of euclidean Artin groups is finally clarified.
\end{abstract}


\maketitle

There is an irreducible Artin group of euclidean type for each of the
extended Dynkin diagrams.  In particular, there are four infinite
families, $\wt A_n$, $\wt B_n$, $\wt C_n$ and $\wt D_n$, known as the
\emph{classical} types plus five remaining \emph{exceptional} examples
$\wt E_8$, $\wt E_7$, $\wt E_6$, $\wt F_4$, and $\wt G_2$.  Most of
these groups have been poorly understood until very recently.  Among
the few known results are a clarification of the structure of the
Artin groups of types $\wt A_2$, $\wt C_2$ and $\wt G_2$ by Craig
Squier in \cite{Squier87} and two papers by Fran\c{c}ois Digne
\cite{Digne06,Digne12} proving that the Artin groups of type $\wt A_n$
and $\wt C_n$ have dual presentations that are Garside structures.
Our main result is that Squier's and Digne's examples are the only
ones that have dual presentations that are Garside structures.

\begin{main}[Dual presentations and Garside structures]\label{main:dual}
  The unique dual presentation of $\art(\wt X_n)$ is a Garside
  structure when $X$ is $C$ or $G$ and it is not a Garside structure
  when $X$ is $B$, $D$, $E$ or $F$.  When the group has type $A$ there
  are distinct dual presentations and the one investigated by Digne is
  the only one that is a Garside structure.
\end{main}

The proof is made possible by a simple combinatorial model developed
in collaboration with Noel Brady that encodes all minimal length
factorizations of a euclidean isometry into reflections
\cite{BrMc-factor}.  Although the results here are essentially
negative, they establish the foundations for positive results
presented in \cite{McSu-artin-euclid}.  In this subsequent article a
new class of Garside groups are constructed from crystallographic
groups closely related to euclidean Coxeter groups and these new
groups contain euclidean Artin groups as subgroups, thereby clarifying
their algebraic structure.  For a survey of all three articles see
\cite{Mc-artin-survey}.

The article is structured as follows.  The first sections give basic
definitions, define dual Artin groups and dual presentations, and
review the results of \cite{BrMc-factor} on factoring euclidean
isometries into reflections.  The middle sections apply these results
to understand how Coxeter elements of irreducible euclidean Coxeter
groups can be factored into reflections present in the group, thereby
constructing dual presentations.  The final sections record explicit
results on a type-by-type basis, from which the main theorem
immediately follows.  As a final note, I would like to highlight the
fact that the rough outline of the main theorem was established in
collaboration with John Crisp several years ago while I was visiting
him in Dijon.  John has since left mathematics but his central role in
the genesis of this work needs to be acknowledged.

%%%%%%%%%%%%%%%%%%%%%%%%%%%
\section{Basic definitions}
%%%%%%%%%%%%%%%%%%%%%%%%%%%

This short section provides some basic definitions that are included
for completeness.  The terminology roughly follows \cite{DaPr02},
\cite{Humphreys90} and \cite{EC1}.

\begin{defn}[Coxeter groups]\label{def:coxeter}
  A \emph{Coxeter group} is any group $W$ that can be defined by a
  presentation of the following form.  It has a standard finite
  generating set $S$ and only two types of relations.  For each $s\in
  S$ there is a relation $s^2=1$ and for each unordered pair for
  distinct elements $s,t\in S$ there is at most one relation of the
  form $(st)^m=1$ where $m = m(s,t) > 1$ is an integer.  When no
  relation involving $s$ and $t$ occurs we consider $m(s,t)=\infty$.
  A \emph{reflection} in $W$ is any conjugate of an element of $S$ and
  we use $R$ to denote the set of all reflections in $W$.  In other
  words, $R=\{ wsw^{-1} \mid s \in S, w \in W\}$.  This presentation
  is usually encoded in a labeled graph $\Gamma$ called a
  \emph{Coxeter diagram} with a vertex for each $s\in S$, an edge
  connecting $s$ and $t$ if $m(s,t)>2$ and a label on this edge if
  $m(s,t)>3$.  The group defined by the presentation encoded in
  $\Gamma$ is denoted $W = \cox(\Gamma)$.  A Coxeter group is
  \emph{irreducible} when its diagram is connected.
\end{defn}

\begin{defn}[Artin groups]\label{def:artin}
  For each Coxeter diagram $\Gamma$ there is an \emph{Artin group}
  $\art(\Gamma)$ defined by a presentation with a relation for each
  two-generator relation in the standard presentation of
  $\cox(\Gamma)$.  More specifically, if $(st)^m=1$ is a relation in
  $\cox(\Gamma)$ then the presentation of $\art(\Gamma)$ has a
  relation that equates the two length~$m$ words that strictly
  alternate between $s$ and $t$.  Thus $(st)^2 = 1$ becomes $st=ts$,
  $(st)^3=1$ becomes $sts=tst$, $(st)^4=1$ becomes $stst=tsts$, etc.
  There is no relation when $m(s,t)$ is infinite.
\end{defn}

\begin{defn}[Posets]\label{def:poset}
  Let $P$ be a partially ordered set.  If $P$ contains both a minimum
  element and a maximum element then it is \emph{bounded}.  For each
  $Q\subset P$ there is an induced \emph{subposet} structure on $Q$ by
  restricting the partial order on $P$.  A subposet $C$ in which any
  two elements are comparable is called a \emph{chain} and its
  \emph{length} is $\vert C \vert -1$.  Every finite chain is bounded
  and its maximum and minimum elements are its \emph{endpoints}.  If a
  finite chain $C$ is not a subposet of a strictly larger finite chain
  with the same endpoints, then $C$ is \emph{saturated}.  Saturated
  chains of length~$1$ are called \emph{covering relations}.  If every
  saturated chain in $P$ between the same pair of endpoints has the
  same finite length, then $P$ is \emph{graded}.  The \emph{rank} of
  an element $p$ is the length of the longest chain with $p$ as its
  upper endpoint and its \emph{corank} is the length of the longest
  chain with $p$ as its lower endpoint, assuming such chains exists.
  The \emph{dual} $P^*$ of a poset $P$ has the same underlying set but
  the order is reversed, and a poset is \emph{self-dual} when it and
  its dual are isomorphic.
\end{defn}

\begin{defn}[Lattices]\label{def:lattice}
  Let $Q$ be any subset of a poset $P$.  A lower bound for $Q$ is any
  $p\in P$ with $p \leq q$ for all $q\in Q$.  When the set of lower
  bounds for $Q$ has a maximum element, this element is the
  \emph{greatest lower bound} or \emph{meet} of $Q$.  Upper bounds and
  the \emph{least upper bound} or \emph{join} of $Q$ are defined
  analogously.  The meet and join of $Q$ are denoted $\bigmeet Q$ and
  $\bigjoin Q$ in general and $u \meet v$ and $u \join v$ if $u$ and
  $v$ are the only elements in $Q$.  When every pair of elements has a
  meet and a join, $P$ is a \emph{lattice} and when every subset has a
  meet and a join, it is a \emph{complete lattice}.
\end{defn}

\begin{figure}
  \begin{tikzpicture}
    \coordinate (one) at (0,1.2);
    \coordinate (zero) at (0,-1.2);
    \coordinate (a) at (-.6,.4);
    \coordinate (b) at (.6,.4);
    \coordinate (c) at (-.6,-.4);
    \coordinate (d) at (.6,-.4);
    \draw[-,thick] (b)--(one)--(a)--(c)--(zero)--(d);
    \draw[-,thick] (a)--(d)--(b)--(c);
    \fill (one) circle (.6mm) node[anchor=south] {$1$};
    \fill (a) circle (.6mm) node[anchor=east] {$a$};
    \fill (b) circle (.6mm) node[anchor=west] {$b$};
    \fill (c) circle (.6mm) node[anchor=east] {$c$};
    \fill (d) circle (.6mm) node[anchor=west] {$d$};
    \fill (zero) circle (.6mm) node[anchor=north] {$0$};
  \end{tikzpicture}
  \caption{A bounded graded poset that is not a
    lattice.\label{fig:non-lattice}}
\end{figure}
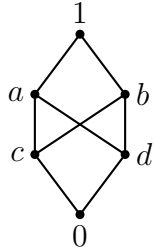

\begin{defn}[Bowties]\label{def:bowtie}
  Let $P$ be a poset.  A \emph{bowtie} in $P$ is a $4$-tuple of
  distinct elements $(a,b:c,d)$ such that $a$ and $b$ are minimal
  upper bounds for $c$ and $d$ and $c$ and $d$ are maximal lower
  bounds for $a$ and $b$.  The name reflects the fact that when edges
  are drawn to show that $a$ and $b$ are above $c$ and $d$, the
  configuration looks like a bowtie.  See
  Figure~\ref{fig:non-lattice}.  It turns out that a bounded graded
  poset $P$ is a lattice iff $P$ contains no bowties \cite{BrMc10}.
  This makes it easy to show that certain subposets are also not
  lattices.  For example, if $P$ is not a lattice because it contains
  a bowtie $(a,b:c,d)$ and $Q$ is any subposet that contains all four
  of these elements, then $Q$ is also not a lattice since it contains
  the same bowtie.
\end{defn}

%%%%%%%%%%%%%%%%%%%%%%%%%%%
\section{Dual Artin groups}
%%%%%%%%%%%%%%%%%%%%%%%%%%%

As mentioned in the introduction, attempts to understand Artin groups
of euclidean type using standard techniques have had limited success.
In this article and its sequel significant progress is made through
the use of dual presentations for Artin groups defined using intervals
in Coxeter groups \cite{McSu-artin-euclid}.

\begin{defn}[Intervals in metric spaces]
  Let $x$, $y$ and $z$ be points in a metric space $(X,d)$.  Borrowing
  from euclidean plane geometry we say that $z$ is \emph{between} $x$
  and $y$ whenever the triangle inequality degenerates into an
  equality.  Concretely $z$ is between $x$ and $y$ when $d(x,z) +
  d(z,y) = d(x,y)$.  The \emph{interval} $[x,y]$ is the collection of
  points between $x$ and $y$ and this includes both $x$ and $y$.
  Intervals can also be endowed with a partial ordering by declaring
  that $u \leq v$ whenever $d(x,u) + d(u,v) + d(v,y) = d(x,y)$.
\end{defn}

Fixing a generating set for a group defines a natural metric and this
leads to the notion of an interval in a group.

\begin{defn}[Intervals in groups]
  A \emph{marked group} is a group $G$ with a fixed generating set $S$
  which, for convenience, we assume is symmetric and injects into $G$.
  The (right) \emph{Cayley graph of $G$ with respect to $S$} is a
  labeled directed graph denoted $\cay(G,S)$ with vertices indexed by
  $G$ and edges indexed by $G \times S$.  The edge $e_{(g,s)}$ has
  \emph{label} $s$, it starts at $v_g$ and ends at $v_{g'}$ where $g'
  = g\cdot s$. There is a natural faithful, vertex-transitive, label
  and orientation preserving left action of $G$ on its Cayley graph
  and these are the only graph automorphisms that preserves labels and
  orientations.  The \emph{distance} $d(g,h)$ is the combinatorial
  length of the shortest path in the Cayley graph from $v_g$ to $v_h$
  and note that the symmetry assumption allows us to restrict
  attention to directed paths.  This defines a metric on $G$ and from
  this metric we get intervals.  More explicitly, for $g,h \in G$, the
  \emph{interval} $[g,h]$ is the poset of group elements between $g$
  and $h$ with $g' \in [g,h]$ when $d(g,g') + d(g',h) = d(g,h)$ and
  $g' \leq g''$ when $d(g,g') + d(g',g'') + d(g'',h) = d(g,h)$.
\end{defn}

The interval $[g,h]$ is a bounded graded poset whose Hasse diagram is
embedded as a subgraph of the Cayley graph $\cay(G,S)$ as the union of
all minimal length directed paths from $v_g$ to $v_h$.  This is
because $g'\in [g,h]$ means $v_{g'}$ lies on some minimal length path
from $v_g$ to $v_h$ and $g' < g''$ means that $v_{g'}$ and $v_{g''}$
both occur on a common minimal length path from $v_g$ to $v_h$ with
$v_{g'}$ occurring before $v_{g''}$.  Because the structure of a
graded poset can be recovered from its Hasse diagram, we let $[g,h]$
denote the edge-labeled directed graph that is visible inside
$\cay(G,S)$.  The left action of a group on its right Cayley graph
preserves labels and distances.  Thus the interval $[g,h]$ is
isomorphic (as a labeled oriented directed graph) to the interval
$[1,g^{-1}h]$.  In other words, every interval in the Cayley graph of
$G$ is isomorphic to one that starts at the identity.  We call
$g^{-1}h$ the \emph{type} of the interval $[g,h]$ and note that
intervals are isomorphic iff they have the same type.

\begin{defn}[Distance order]\label{def:iso-int}
  The \emph{distance order} on a marked group $G$ is defined by
  setting $g' \leq g$ iff $g' \in [1,g]$.  This turns $G$ into a poset
  that contains an interval of every type that occurs in the metric
  space on $G$.  Next, there is a \emph{length function} $\ls\colon
  G\to \N$ that sends each element to its distance from the identity.
  The value $\ls(g) = d(1,g)$ is called the \emph{$S$-length of $g$}
  and it is also the length of the shortest factorization of $g$ in
  terms of elements of $S$.  Because Cayley graphs are homogeneous,
  metric properties of the distance function translate into properties
  of $\ls$.  Symmetry and the triangle inequality, for example, imply
  that $\ls(g) = \ls(g^{-1})$, and $\ls(gh)\leq \ls(g) + \ls(h)$.
\end{defn}

Intervals in groups can be used to construct new groups.

\begin{defn}[Interval groups]\label{def:interval-gps}
  Let $G$ be a group generated by a set $S$ and let $g$ and $h$ be
  distinct elements in $G$.  The \emph{interval group} $G_{[g,h]}$ is
  defined as follows.  Let $S_0$ be the elements of $S$ that actually
  occurs as labels of edges in $[g,h]$.  The group $G_{[g,h]}$ has
  $S_0$ as its generators and we impose all relations that are visible
  as closed loops inside the portion of the Cayley graph of $G$ that
  we call $[g,h]$.  The elements in $S \setminus S_0$ are not included
  since they do not occur in any relation.  More precisely, if they
  were included as generators, they would generate a free group that
  splits off as a free factor.  Thus it is sufficient to understand
  the group defined above.  Next note that this group structure only
  depends on the type of the interval so it is sufficient to consider
  interval groups of the form $G_{[1,g]}$.  For these groups we
  simplify the notation to $G_g$ and say that $G_g$ is the interval
  group obtained by \emph{pulling $G$ apart at $g$}. 
\end{defn}

The interval $[1,g]$ incorporates all of the essential information
about the presentation of $G_g$.  More traditional presentations for
interval groups using relations are established in
\cite{McCammond-cont-braids} and described in the next section.  Dual
Artin groups are examples of interval groups.

\begin{defn}[Dual Artin groups]\label{def:dual-artin}
  Let $W = \cox(\Gamma)$ be a Coxeter group with standard generating
  set $S$ and reflections $R$.  For any fixed total ordering of the
  elements of $S$, the product of these generators in this order is
  called a \emph{Coxeter element} and for each Coxeter element $w$
  there is a dual Artin group defined as follows.  Let $[1,w]$ be the
  interval in the Cayley graph of $W$ with respect to $R$ and let $R_0
  \subset R$ be the subset of reflections that actually occur in some
  minimal length factorizations of $w$.  The \emph{dual Artin group
    with respect to $w$} is the group $W_w = \dart(\Gamma,w)$
  generated by $R_0$ and subject only to those relations that are
  visible inside the interval $[1,w]$.
\end{defn}

\begin{rem}[Artin groups and dual Artin groups]
  In general the relationship between the Artin group $\art(\Gamma)$
  and the dual Artin group $\dart(\Gamma,w)$ is not yet completely
  clear.  It is straightforward to show using the Tits representation
  that the product of the elements in $S$ that produce $w$ is a
  factorization of $w$ into reflections of minimum length which means
  that this factorization describes a directed path in $[1,w]$.  As a
  consequence $S$ is a subset of $R_0$.  Moreover, the standard Artin
  relations are consequences of relations visible in $[1,w]$ (as
  illustrated in \cite{BrMc00}) so that the injection of $S$ into
  $R_0$ extends to a group homomorphism from $\art(\Gamma)$ to
  $\dart(\Gamma,w)$.  When this homomorphism is an isomorphism, we say
  that the interval $[1,w]$ encodes a \emph{dual presentation} of
  $\art(\Gamma)$.
\end{rem}

Every dual Artin group that has been successfully analyzed so far is
isomorphic to the corresponding Artin group and as a consequence its
group structure is independent of the Coxeter element $w$ used in its
construction.  It is precisely because this assertion has not been
proved in full generality that dual Artin groups deserve a separate
name.  One reason that dual Artin groups are of interest is that they
nearly satisfy the requirements to be Garside groups.  In fact, from
the construction it is easy to show that interval $[1,w]$ used to
define a dual Artin group has all of the properties of a Garside
structure with one exception. 

\begin{prop}[Garside structures]\label{prop:lattice-garside}
  Let $\Gamma$ be a Coxeter diagram and let $w$ be a Coxeter element
  for $W = \cox(\Gamma)$.  If the interval $[1,w]$ is a lattice, then
  the dual Artin group $\dart(\Gamma,w)$ is a Garside
  group.
\end{prop}

The reader should note that we are using ``Garside structure'' and
``Garside group'' in the expanded sense of Digne
\cite{Digne06,Digne12} rather than the original definition that
requires the generating set to be finite.  In the language of the
``Foundations of Garside theory'' book \cite{DDGKM-garside} these are
``quasi-Garside'' groups and structures.  Since these are the only
types of Garside structures considered here, the prefix ``quasi'' is
dropped but we shall occasionally remind the reader that the interval
$[1,w]$ has infinitely many elements.  The grading of the interval
used to define an interval group substitutes for finiteness of the
generating set in forcing algorithmic processes to terminate.  The
standard proofs are otherwise unchanged.  With the exception of the
shift from finite to infinite generating sets,
Proposition~\ref{prop:lattice-garside} was stated by David Bessis in
\cite[Theorem~0.5.2]{Be03}.  For a more detailed discussion see
\cite{Be03} and particularly the book \cite{DDGKM-garside}.  Interval
groups appear in \cite{DehornoyDigneMichel13} and in \cite[Chapter
  VI]{DDGKM-garside} as the ``germ derived from a groupoid''.  The
terminology is different but the translation is straightforward.

Consequences of being a Garside group include normal forms for
elements and a finite-dimensional classifying space, which imply that
the group has a decidable word problem and is torsion-free
\cite{ChMeWh04,DePa99}.  It was an early hope that every dual Artin
group would be a Garside group, but this article provides the first
explicit examples where this hope fails.  Concretely, when $w$ is
Coxeter element for an irreducible euclidean Coxeter group $W$, we
show that the interval $[1,w]$ in $\cay(W,R)$ is not a lattice except
for the cases already analyzed by Squier and Digne.

%%%%%%%%%%%%%%%%%%%%%%%%%%%%
\section{Dual presentations}
%%%%%%%%%%%%%%%%%%%%%%%%%%%%

The section records some known results about presentations for
interval groups in general and for dual Artin groups in particular.

\begin{defn}[Factorizations]
  In a group $G$ generated by a set $S$, each positive word over $S$
  can be evaluated as a group element and the word \emph{represents}
  the element $g$ to which it evaluates.  In the Cayley graph
  $\cay(G,S)$ a word represents $g$ iff the unique directed path that
  starts at $v_1$ and corresponds to the word ends at the vertex
  $v_g$.  An element is \emph{positive} if it is represented by some
  positive word.  A \emph{minimal positive factorization} of a
  positive element $g$ is a word corresponding to a minimal length
  directed path from $v_1$ to $v_g$ in the Cayley graph $\cay(G,S)$.
  Minimal positive factorizations are called \emph{reduced
    $S$-decompositions} in \cite{Be03}.  Inside an interval every
  directed path corresponds to a positive word whose length is equal
  to the distance between its endpoints.  In particular, there is a
  bijective correspondence between directed paths in $[1,g]$ from
  $v_1$ and $v_g$ and minimal positive factorizations of $g$.
\end{defn}

\begin{defn}[Bigons]
  Let $g$ be a positive element in an $S$-group $G$ and consider two
  directed paths in $[1,g]$ that start at the same vertex and end at
  the same vertex.  The relation $U=V$ that equates the positive words
  $U$ and $V$ corresponding to these paths is called a \emph{bigon
    relation} and it holds in $G_g$ since $UV^{-1}$ is visible as a
  closed loop inside $[1,g]$.  Both positive words necessarily have
  the same length $k$ which we call the \emph{height} of the relation.
  A bigon relation is \emph{big} when its height is $k = d(1,g)$,
  i.e. as big as possible and it is \emph{small} when it is not a
  consequence of those bigon relations of strictly shorter height.
\end{defn}

The various types of bigon relations are sufficient to define interval
groups.  A detailed proof of the following proposition can be found in
\cite{McCammond-cont-braids} but we include a brief version for
completeness.

\begin{prop}[Bigon presentations]\label{prop:bigons}
  If $g$ is a positive element in a group $G$ generated by a set $S$,
  $S_0$ is the subset of $S$ labeling edges in $[1,g]$, and
  $\mathcal{R}_a$, $\mathcal{R}_b$ and $\mathcal{R}_s$ denote the
  collection of all bigon, big bigon and small bigon relations,
  respectively, visible in the interval $[1,g]$, then $\langle S_0
  \mid \mathcal{R}_a \rangle$, $\langle S_0 \mid \mathcal{R}_b
  \rangle$ and $\langle S_0 \mid \mathcal{R}_s \rangle$ are three
  presentations of $G_g$.
\end{prop}

\begin{proof}
  Since $\mathcal{R}_a$, $\mathcal{R}_b$ and $\mathcal{R}_s$ are all
  subsets of the relations visible inside the interval $[1,g]$ it is
  sufficient to show that the remaining relations are consequences of
  these relations.  First, given any closed undirected path in the
  portion of the Cayley graph that is $[1,g]$, one can add a path from
  $v_1$ to each vertex and show that this loop is a consequence of
  bigon relations.  More explicitly, draw the closed loop as simple
  loop in the plane, place $v_1$ in the center and the paths to the
  vertices as subdivided radial line segments.  Every complementary
  region is then a bigon relation.  This shows that $\langle S_0 \mid
  \mathcal{R}_a \rangle$ is a presentation for $G_g$.  Next, by
  extending each bigon with paths from $v_1$ to the start point and
  from the endpoint to $v_g$, it is clear that every bigon is a
  consequence via cancellation of a big bigon relation.  And finally,
  the small bigon relations are, by definition, sufficient to
  establish all big bigon relations.
\end{proof}

When the generating set is closed under conjugation -- as is the case
with the reflections inside a Coxeter group -- there are many bigon
relations of height~$2$ visible in any interval.

\begin{defn}[Hurwitz action]
  If $S$ is any subset of a group $G$ that is closed under conjugation
  and $S^n$ denotes all words of length $n$ over $S$, then there is a
  natural action of the $n$-strand braid group on $S^n$.  The standard
  braid generator $s_i$ replaces the two letter subword $ab$ in
  positions $i$ and $i+1$ with the subword $ca$ where $c = aba^{-1}
  \in S$ and it leaves the letters in the other positions unchanged.
  It is straightforward to check that this action satisfies the
  relations in the standard presentation of the braid group.  Of
  particular interest here is that every word in the same orbit under
  this action evaluates to the same element of $G$ and thus there is a
  well-defined \emph{Hurwitz action} of the $k$-strand braid group on
  the minimal positive factorizations of an element $g$ where $k =
  d(1,g)$.
\end{defn}

Notice that when a standard braid generator replaces $ab$ with $ca$
inside a minimal positive factorization of $g$, $ab=ca$ is a
height~$2$ bigon relation visible in $[1,g]$.  It thus makes sense to
call any height~$2$ bigon relation of the form $ab=ca$ visible inside
$[1,g]$ a \emph{Hurwitz relation}.  Relations of this form are what
Bessis calls \emph{dual braid relations} in \cite{Bessis-top}.  When
the Hurwitz action is transitive on factorizations, these relations
are sufficient to define $G_g$.

\begin{prop}[Hurwitz presentations]\label{prop:hurwitz}
  Let $g$ be a positive element in a group $G$ generated by $S$ and
  let $S_0$ be the subset of $S$ labeling edges in $[1,g]$.  If the
  Hurwitz action is transitive on minimal positive factorizations of
  $g$, then $\langle S_0 \mid \mathcal{R}_h \rangle$ is a presentation
  of $G_g$ where $\mathcal{R}_h$ denotes the collection of Hurwitz
  relations visible in $[1,g]$.
\end{prop}

\begin{proof}
  Let $H$ be the group defined in the statement of the proposition.
  Since the Hurwitz relations are visible in $[1,g]$ they are
  satisfied by $G_g$ and every relation that holds in $H$ also holds
  in $G_g$.  On the other hand, the transitivity of the action implies
  that every big bigon relation is a consequence of Hurwitz relations
  and by Proposition~\ref{prop:bigons} these are sufficient to define
  $G_g$.  Thus every relation that holds in $G_g$ holds in $H$ and the
  two groups are the same.
\end{proof}

Although their results are often stated in a completely different
language, representation theorists have already addressed the question
of whether or not the Hurwitz action is transitive on minimal length
reflection factorizations of a Coxeter element in a Coxeter group in
many contexts.  For example, Crawley-Boevey has shown that
transitivity holds when every $m = m(s,t)$ is either $2$, $3$, or
$\infty$ \cite{Crawley-Boevey92} and Ringel has extended this to
include the crystallographic cases, i.e. where every $m = m(s,t)$ is
$2$, $3$, $4$, $6$, or $\infty$ \cite{Ringel94}.  In 2010 Igusa and
Schiffler in \cite{IgusaSchiffler10} proved transitivity of the
Hurwitz action for all Coxeter groups in complete generality and in
2014 a short proof of this general fact was posted by Baumeister,
Dyer, Stump and Wegener \cite{BDSW-hurwitz}.  See also
\cite{IngallsThomas09} and \cite{Igusa11}.  As a consequence
Proposition~\ref{prop:hurwitz} applies to all of the dual Artin groups
of euclidean type and we can use this to establish that dual euclidean
Artin groups and euclidean Artin groups are isomorphic.  A proof of
Theorem~\ref{thm:dart-art} will be included as an appendix in the
final paper in this series \cite{McSu-artin-euclid}.

\begin{thm}[Dual Artin groups are Artin groups]\label{thm:dart-art}
  For every choice of Coxeter element $w$ in an irreducible euclidean
  Coxeter group $W = \cox(\wt X_n)$, the dual Artin group $\dart(\wt
  X_n,w)$ is naturally isomorphic to the Artin group $\art(\wt X_n)$.
\end{thm}

We conclude this section by noting an elementary consequence of the
Hurwitz action.

\begin{lem}[Rewriting factorizations]\label{lem:rewriting}
  Let $w = r_1 r_2 \cdots r_k$ be a reflection factorization in a
  Coxeter group $W$. For any selection $1 \leq i_1 < i_2 < \cdots <
  i_j \leq k$ of positions there is a length~$k$ reflection
  factorization of $w$ whose first $j$ reflections are $r_{i_1}
  r_{i_2} \cdots r_{i_j}$ and another length~$k$ reflection
  factorization of $w$ where these are the last $j$ reflections in the
  factorization.
\end{lem}

%%%%%%%%%%%%%%%%%%%%%%%%%%%%%%
\section{Euclidean isometries}
%%%%%%%%%%%%%%%%%%%%%%%%%%%%%%

In order to analyze intervals in euclidean Coxeter groups, we need to
establish some notations for euclidean isometries. As in
\cite{BrMc-factor} and \cite{SnapperTroyer89} we sharply distinguish
between points and vectors.

\begin{defn}[Points and vectors]\label{def:pt-vec}
  Throughout the article, $V$ denotes an $n$-dimensional real vector
  space with a positive definite inner product and $E$ denotes the
  euclidean which is its affine analog where the vector space
  structure of $V$ (in particular the location of the origin) has been
  forgotten.  The elements of $V$ are \emph{vectors} and the elements
  of $E$ are \emph{points}.  We use greek letters for vectors and
  roman letters for points.  There is a uniquely transitive action of
  $V$ on $E$.  Thus, given a point $x$ and a vector $\lambda$ there is
  a unique point $y$ with $x + \lambda = y$ and given two points $x$
  and $y$ there is a unique vector $\lambda$ with $x + \lambda = y$.
  We say that $\lambda$ is the vector from $x$ to $y$.  For any
  $\lambda \in V$, the map $x\mapsto x+\lambda$ is an isometry
  $t_\lambda$ of $E$ that we call a \emph{translation} and note that
  $t_\mu t_\nu = t_{\mu + \nu} = t_\nu t_\mu$ so the set $T_E = \{
  t_\lambda \mid \lambda \in V\}$ is an abelian group.  For any point
  $x \in E$, the map $\lambda \mapsto x+\lambda$ is a bijection that
  identifies $V$ and $E$ but the isomorphism depends on this initial
  choice of a \emph{basepoint} $x$ in $E$.  Lengths of vectors and
  angles between vectors are calculated using the usual formulas and
  distances and angles in $E$ are defined by converting to
  vector-based calculations.
\end{defn}

\begin{defn}[Linear subspaces of $V$]\label{def:lin}
  A \emph{linear subspace} of $V$ is a subset closed under linear
  combination and every subset of $V$ is contained in a unique minimal
  linear subspace called its \emph{span}.  Every subset $U$ has an
  \emph{orthogonal complement} $U^\perp$ consisting of those vectors
  in $V$ orthogonal to all the vectors in $U$.  When $U$ is a linear
  space there is a corresponding orthogonal decomposition $V = U
  \oplus U^\perp$ and the \emph{codimension} of $U$ is the dimension
  of $U^\perp$.  For more general subsets $U^\perp = \spn(U)^\perp$.
  The linear subspaces of $V$ form a bounded graded self-dual complete
  lattice under inclusion that we call $\lin(V)$. The bounding
  elements are clear, the grading is by dimension (in that a
  $k$-dimensional subspace has rank~$k$ and corank~$n-k$), the meet of
  a collection of subspaces is their intersection and their join is
  the span of their union.  And finally, the map sending a linear
  subspace to its orthogonal complement is a bijection that
  establishes self-duality.
\end{defn}

\begin{defn}[Affine subspaces of $E$]\label{def:aff}
  An \emph{affine subspace} of $E$ is any subset $B$ that contains
  every line determined by distinct points in $B$ and every subset of
  $E$ is contained in a unique minimal affine subspace called its
  \emph{affine hull}.  Associated with any affine subspace $B$ is its
  (linear) \emph{space of directions} $\dir(B) \subset V$ consisting
  of the collection of vectors connecting points in $B$.  The
  \emph{dimension} and \emph{codimension} of $B$ is that of its space
  of directions.  The \emph{affine subspaces} of $E$ partially ordered
  by inclusion form a poset we call $\aff(E)$.  It is a graded poset
  that is bounded above but not below since distinct points are
  distinct minimal elements.  It is neither self-dual nor a lattice.
  There is, however, a well-defined rank-preserving poset map $\aff(E)
  \onto \lin(V)$ sending each affine subspace $B$ to its space of
  directions $\dir(B)$.
\end{defn}

\begin{defn}[Standard forms]
  An affine subspace of $V$ is any subspace that corresponds to an
  affine subspace of $E$ under an identification of $V$ and $E$ and
  the subspaces of this form are translations of linear subspaces.  In
  particular, every affine subspace $M$ in $V$ can be written in the
  form $M = t_\mu(U) = U +\mu = \{ \lambda + \mu \mid \lambda \in U\}$
  where $U$ is a linear subspace of $V$.  This representation is not
  unique, since $U + \mu = U$ for all $\mu \in U$, but it can be made
  unique if we insist that $\mu$ to be of minimal length or,
  equivalently, that $\mu$ be a vector in $U^\perp$.  In this case we
  say $U +\mu$ is the \emph{standard form} of $M$.
\end{defn}

The isometries of $E$ form a group $\isom(E)$ and every isometry has
two basic invariants, one in $V$ and the other in $E$.

\begin{defn}[Basic invariants]\label{def:inv}
  Let $w$ be an isometry of $E$.  If $\lambda$ is the vector from $x$
  to $w(x)$ then we say $x$ is \emph{moved by $\lambda$ under $w$}.
  The collection $\mov(w) = \{ \lambda \mid x+\lambda = w(x), x \in
  E\} \subset V$ of all such vectors is the \emph{move-set} of $w$.
  The move-set is an affine subspace and thus has standard form $U
  +\mu$ where $U$ is a linear subspace and $\mu$ is a vector in
  $U^\perp$.  The points in $E$ that are moved by $\mu$ under $w$ are
  those that are moved the shortest distance.  The collection $\ms(w)$
  of all such points is an affine subspace called the \emph{min-set}
  of $w$.  The sets $\mov(w) \subset V$ and $\ms(w) \subset E$ are the
  \emph{basic invariants} of $w$.
\end{defn}

\begin{defn}[Types of isometries]\label{def:types}
  Let $w$ be an isometry of $E$ and let $U+\mu$ be the standard form
  of its move-set $\mov(w)$.  There are points fixed by $w$ iff $\mu$
  is trivial iff $\mov(w)$ is a linear subspace.  Under these
  conditions we say $w$ is \emph{elliptic} and the min-set $\ms(w)$ is
  just the \emph{fix-set} $\fix(w)$ of points fixed by $w$.
  Similarly, $w$ has no fixed points iff $\mu$ is nontrivial iff
  $\mov(w)$ a nonlinear affine subspace of $V$.  Under these
  conditions we say $w$ is \emph{hyperbolic}.
\end{defn}

The names elliptic and hyperbolic come from a tripartite
classification of isometries of nonpositively curved spaces; the third
type, parabolic, does not occur in this context
\cite{BridsonHaefliger99}.  The simplest examples of hyperbolic
isometries are the nontrivial translations as defined in
Definition~\ref{def:pt-vec}.  They can also be characterized as those
isometries whose move-set is a single point or whose min-set is all of
$A$.  The simplest example of an elliptic isometry is a reflection.

\begin{defn}[Reflections]\label{def:reflections}
  A \emph{hyperplane} $H$ in $E$ is an affine subspace of
  codimension~$1$ and there is a unique nontrivial isometry $r$ that
  fixes $H$ pointwise called a \emph{reflection}.  The space of
  directions $\dir(H)$ is a codimension~$1$ linear subspace in $V$ and
  it has a $1$-dimensional orthogonal complement $L$.  The basic
  invariants of $r$ are $\mov(r) = L$ and $\ms(r) = \fix(r) = H$.  The
  set of all reflections is denoted $R_E$.
\end{defn}

%%%%%%%%%%%%%%%%%%%%%%%%
\section{Factorizations}
%%%%%%%%%%%%%%%%%%%%%%%%

This section reviews the structure of intervals in $\isom(E)$ when
viewed as a group generated by the set $R_E$ of all reflections.  In
particular, it introduces the combinatorial models constructed in
\cite{BrMc-factor} that encode the poset structure of these intervals.  The
first thing to note is that the length function with respect to all
reflections is easy to compute using the basic invariants of
isometries, a result known as Scherk's theorem \cite{SnapperTroyer89}.

\begin{thm}[Reflection length]\label{thm:lr}
  The reflection length of an isometry is determined by its basic
  invariants.  More specifically, if $w$ is an isometry of $E$ whose
  move-set is $k$-dimensional, then $\lre(w) = k$ when $w$ is
  elliptic, and $\lre(w) = k+2$ when $w$ is hyperbolic.
\end{thm}

Next, consider the following combinatorially defined posets.

\begin{defn}[Model posets]\label{def:model}
  We construct a global poset $P$ from two types of elements.  For
  each nonlinear affine subspace $M$ in $V$, $P$ contains a
  \emph{hyperbolic} element $h^M$ and for each affine subspace $B$ in
  $A$, $P$ contains an \emph{elliptic} element $e^B$.  We also define
  an \emph{invariant map} $\inv \colon \isom(E) \onto P$ that sends
  $w$ to $h^{\mov(w)}$ when $w$ is hyperbolic and to $e^{\fix(w)}$
  when $w$ is elliptic.  This explains the names and the notation.
  The elements of $P$ are ordered as follows.  First, hyperbolic
  elements are ordered by inclusion and elliptic elements by reverse
  inclusion: $h^M \leq h^{M'}$ iff $M \subset M'$ and $e^B \leq
  e^{B'}$ iff $B \supset B'$.  Next, no elliptic element is ever above
  a hyperbolic element.  And finally, $e^B < h^M$ iff $M^\perp \subset
  \dir(B)$. Note, however, that since $M$ is nonlinear, the vectors
  orthogonal to all of $M$ are also orthogonal to its span, a linear
  subspace whose dimension is $\dim(M) + 1$.  Transitivity is an easy
  exercise.  It was shown in \cite{BrMc-factor} that when $\isom(E)$
  is viewed as a marked group generated by the set $R_E$ of all
  reflections and viewed as a poset under the distance order, the
  invariant map is a rank-preserving order-preserving map from
  $\isom(E)$ to $P$.  As a consequence, for any isometry $w$ the
  invariant map sends isometries in $[1,w]$ to elements less than or
  equal to $\inv(w)$.  Let $P(w)$ denote the subposet of $P$ induced
  by restricting to those elements less than or equal to $\inv(w)$ and
  call $P(w)$ the \emph{model poset for $w$}.
\end{defn}

The following is the main theorem proved in \cite{BrMc-factor}.

\begin{thm}[Model posets]\label{thm:model}
  For each isometry $w$, the invariant map establishes a poset
  isomorphism between the interval $[1,w]$ and the model poset $P(w)$.
  As a consequence, the minimum length reflection factorizations of
  $w$ are in bijection with the maximal chains in $P(w)$.
\end{thm}

Theorem~\ref{thm:model} is in sharp contrast with the non-injectivity
of the invariant map in general. There are, for example, many
different rotations that fix the same codimension~$2$ subspace.  Since
it is useful to have a notation for model subposets in the absence of
an isometry, let $P^M$ denote the subposet of $P$ induced by
restricting to those elements less than or equal to $h^M$ for a
nonlinear affine subspace $M \subset V$ and let $P^B$ denote the
subposet induced by restricting to those elements less than or equal
to $e^B$ for an affine subspace $B \subset E$.  Thus $P(w) =
P^{\mov(w)}$ when $w$ is hyperbolic and $P(w) = P^{\fix(w)}$ when $w$
is elliptic.  Note that this notation is not ambiguous because $M$ and
$B$ are subsets of different spaces.

\begin{rem}[Auxillary results]\label{rem:aux}
  In the process of proving Theorem~\ref{thm:model} many auxillary
  results are established in \cite{BrMc-factor} that are useful here.  For
  example, the direction space of the min-set of an isometry is the
  orthogonal complement of the direction space of its move-set
  \cite[Lemma~$\LemOrthInv$]{BrMc-factor}.  In symbols, when $\mov(w) = U+\mu$
  in standard form, $\dir(\mov(w)) = U$ and $\dir(\ms(w)) = U^\perp$.
  Next, when $w$ is an isometry and $r$ is a reflection, the move-set
  of $w$ and the move-set of $rw$ are nested so that one is a
  codimension~$1$ subspace of the other
  \cite[Proposition~$\PropMoveRefl$]{BrMc-factor}.  As a consequence, the
  dimension of the min-set also changes by exactly one dimension in
  the opposite direction.  A third useful result identifies the
  min-set of a hyperbolic isometry as the unique affine subspace $B$
  in $E$ that is stabilized by $w$ of the correct dimension and where
  all points undergo the same motion
  \cite[Proposition~$\PropIdentMin$]{BrMc-factor}.
\end{rem}

One final result that we need from \cite{BrMc-factor} is a
characterization of exactly when hyperbolic posets are not lattices
and the explicit locations of the bowties that bear witness to this
fact.

\begin{thm}[Hyperbolic posets are not lattices]\label{thm:hyp-bowtie}
  Let $M$ be a nonlinear affine subspace of $V$.  The poset $P^M$
  contains a bowtie and is not a lattice iff $\dir(M)$ contains a
  proper non-trivial linear subspace $U$, which is true iff the
  dimension of $M$ is at least $2$.  More precisely, for every such
  subspace and for every choice of distinct elements $h^{M_1}$ and
  $h^{M_2}$ with $\dir(M_1) = \dir(M_2) = U$ and distinct elements
  $e^{B_1}$ and $e^{B_2}$ with $\dir(B_1) = \dir(B_2) = U^\perp$,
  these four elements form a bowtie.  Conversely, all bowties in $P^M$
  are of this form.
\end{thm}

%%%%%%%%%%%%%%%%%%%%%%%%%%%%%%%%%%
\section{Euclidean Coxeter groups}
%%%%%%%%%%%%%%%%%%%%%%%%%%%%%%%%%%

The irreducible Coxeter groups of interest in this article are those
that naturally act geometrically, i.e. properly discontinuously and
cocompactly by isometries, on a euclidean space with its generators
acting as reflections.  Their classification is well-known and they
are described by the Coxeter diagrams known as the \emph{extended
  Dynkin diagrams}.  There are four infinite families and five
sporadic examples of such diagrams and they are displayed in
Figure~\ref{fig:dynkin}.  In this restricted context, it is
traditional to replace edges labeled $4$ and $6$ with double and
triple edges, respectively.  The white vertex and the orientations on
the double and triple edges are explained below.  This section records
basic facts about these groups with the Coxeter group of type $\wt
G_2$ used to illustrate the concepts under discussion.  For additional
details see \cite{Humphreys90}.

\begin{figure}
  \begin{tabular}{cc}
    \begin{tikzpicture}[scale=.8]
    \begin{scope}[yshift=5.25cm,xshift=2cm]
      \node at (-3,0) {$\wt A_1$};
      \drawInfEdge{(0,0)}{(1,0)}
      \drawSpeDot{(1,0)}
      \drawESpDot{(0,0)}
    \end{scope}
    \begin{scope}[yshift=4cm,xshift=.5cm]
      \node at (-1.5,0) {$\wt A_n$};
      \drawEdge{(0,0)}{(2,0)}
      \drawDashEdge{(0,0)}{(2,.7)}
      \drawDashEdge{(4,0)}{(2,.7)}
      \drawEdge{(3,0)}{(4,0)}
      \drawDotEdge{(2,0)}{(3,0)}
      \foreach \x in {0,1,2,4} {\drawRegDot{(\x,0)}}
      \drawSpeDot{(2,.7)}
      \drawESpDot{(3,0)}
    \end{scope}
    \begin{scope}[yshift=3cm]
      \node at (-1,0) {$\wt C_n$};
      \drawDoubleEdge{(1,0)}{(0,0)}{(0,.1)}
      \drawEdge{(1,0)}{(2,0)}
      \drawDotEdge{(2,0)}{(3,0)}
      \drawEdge{(3,0)}{(4,0)}
      \drawDashDoubleEdge{(4,0)}{(5,0)}{(0,.1)}
      \foreach \x in {1,2,3,4} {\drawRegDot{(\x,0)}}
      \drawESpDot{(0,0)}
      \drawSpeDot{(5,0)}
    \end{scope}
    \begin{scope}[yshift=1.8cm]
      \node at (-1,0) {$\wt B_n$};
      \drawDoubleEdge{(0,0)}{(1,0)}{(0,.1)}
      \drawEdge{(1,0)}{(2,0)}
      \drawDotEdge{(2,0)}{(3,0)}
      \drawEdge{(3,0)}{(4,0)}
      \drawDashEdge{(4,0)}{(4.707,.707)}
      \drawEdge{(4,0)}{(4.707,-.707)}
      \drawRegDot{(4.707,-.707)}
      \foreach \x in {0,2,3,4} {\drawRegDot{(\x,0)}}
      \drawESpDot{(1,0)}
      \drawSpeDot{(4.707,.707)}
    \end{scope}
    \begin{scope}
      \node at (-1,0) {$\wt D_n$};
      \drawEdge{(.293,.707)}{(1,0)}
      \drawEdge{(.293,-.707)}{(1,0)}
      \drawEdge{(1,0)}{(2,0)}
      \drawDotEdge{(2,0)}{(3,0)}
      \drawEdge{(3,0)}{(4,0)}
      \drawDashEdge{(4,0)}{(4.707,.707)}
      \drawEdge{(4,0)}{(4.707,-.707)}
      \drawRegDot{(.293,.707)}
      \drawRegDot{(.293,-.707)}
      \drawRegDot{(4.707,-.707)}
      \foreach \x in {2,3,4} {\drawRegDot{(\x,0)}}
      \drawESpDot{(1,0)}
      \drawSpeDot{(4.707,.707)}
    \end{scope}
  \end{tikzpicture}
  &
  \begin{tikzpicture}[scale=.8]
  \begin{scope}[yshift=5.8cm,xshift=1cm]
    \node at (-2,0) {$\wt G_2$};
    \drawTripleEdge{(0,0)}{(1,0)}{(0,.1)}
    \drawDashEdge{(1,0)}{(2,0)}
    \drawRegDot{(0,0)}
    \drawESpDot{(1,0)}
    \drawSpeDot{(2,0)}
  \end{scope}
  \begin{scope}[yshift=5.1cm]
    \node at (-1,0) {$\wt F_4$};
    \drawDoubleEdge{(1,0)}{(2,0)}{(0,.1)}
    \drawEdge{(0,0)}{(1,0)}
    \drawEdge{(2,0)}{(3,0)}
    \drawDashEdge{(3,0)}{(4,0)}
    \foreach \x in {0,1,3} {\drawRegDot{(\x,0)}}
    \drawESpDot{(2,0)}
    \drawSpeDot{(4,0)}
  \end{scope}
  \begin{scope}[yshift=3.4cm]
    \node at (-1,0) {$\wt E_6$};
    \drawEdge{(0,0)}{(4,0)}
    \drawEdge{(2,0)}{(2,1)}
    \drawDashEdge{(2,1)}{(3,1)}
    \foreach \x in {0,1,3,4} {\drawRegDot{(\x,0)}}
    \drawRegDot{(2,1)}
    \drawESpDot{(2,0)}
    \drawSpeDot{(3,1)}
  \end{scope}
  \begin{scope}[yshift=1.7cm]
    \node at (-1,0) {$\wt E_7$};
    \drawEdge{(0,0)}{(5,0)}
    \drawEdge{(2,0)}{(2,1)}
    \drawDashEdge{(0,0)}{(0,1)}
    \foreach \x in {0,1,3,4,5} {\drawRegDot{(\x,0)}}
    \drawRegDot{(2,1)}
    \drawESpDot{(2,0)}
    \drawSpeDot{(0,1)}
  \end{scope}
  \begin{scope}
    \node at (-1,0) {$\wt E_8$};
    \drawEdge{(0,0)}{(6,0)}
    \drawDashEdge{(6,0)}{(6,1)}
    \drawEdge{(2,0)}{(2,1)}
    \foreach \x in {0,1,3,4,5,6} {\drawRegDot{(\x,0)}}
    \drawRegDot{(2,1)}
    \drawESpDot{(2,0)}
    \drawSpeDot{(6,1)}
  \end{scope}
  \end{tikzpicture}
  \end{tabular}
  \caption{Four infinite families and five sporadic examples.\label{fig:dynkin}}
\end{figure}
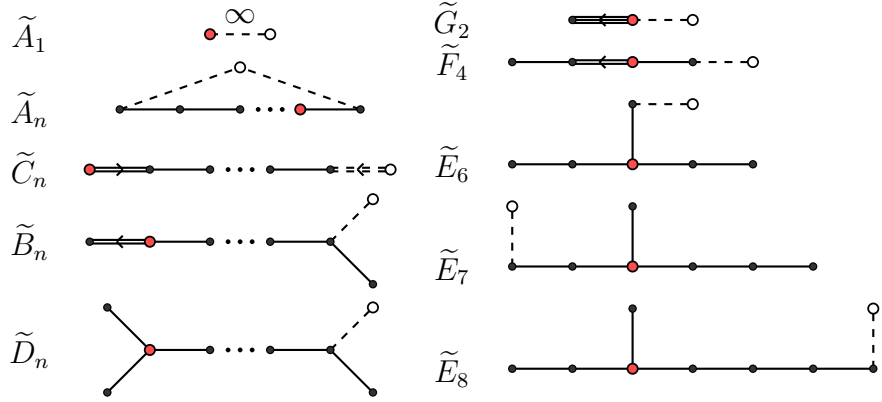

\begin{rem}[Diagrams and simplices]\label{rem:simplices}
  Each extended Dynkin diagram $\Gamma$ is essentially a recipe that
  can be used to reconstruct a euclidean simplex $\sigma$ with
  non-obtuse dihedral angles that are submultiples of $\pi$.  The
  vertices of $\Gamma$ index the outward pointing normal vectors
  $\alpha_s$ to the various facets of $\sigma$ and the label $m =
  m(s,t)$ indicates that the angle between $\alpha_s$ and $\alpha_t$
  is $\pi - \frac{\pi}{m}$ and thus the corresponding dihedral angle
  between their fixed facets is $\frac{\pi}{m}$.  This is sufficient
  information to reconstruct a unique euclidean simplex up to
  similarity.  The $\wt G_2$ diagram in Figure~\ref{fig:dynkin}, for
  example, leads to the construction of a 30-60-90 triangle.
\end{rem}

\begin{figure}
  \includegraphics[scale=1]{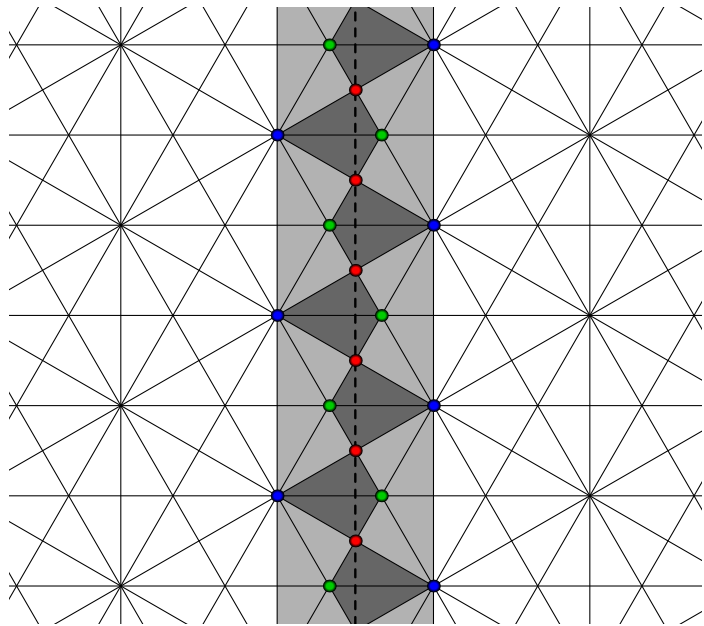}
  \caption{The $\wt G_2$ tiling of the plane with annotations
    corresponding to a particular Coxeter element $w$.  The dashed
    line is the glide axis of $w$, the heavily shaded triangles are
    those for which $w$ is a bipartite Coxeter element, and the
    lightly shaded vertical strip is the convex hull of the vertices
    of these triangles.\label{fig:g2-axis}}
\end{figure}

\begin{defn}[Coxeter complex]\label{def:cox-cplx}
  Let $W$ be an irreducible euclidean Coxeter group with extended
  Dynkin diagram $\Gamma$ and let $\sigma$ be a corresponding
  euclidean simplex.  If we embed $\sigma$ in a euclidean space $E$ so
  that $E$ is the affine hull of $\sigma$ and let $S$ be the
  isometries of $E$ that fix (the affine hull of) one of the facets of
  $\sigma$ pointwise then these reflections generate a group of
  isometries naturally isomorphic to the Coxeter group $W$ with
  standard generating set $S$.  More precisely, using the orbit of
  $\sigma$ under this group action, it is possible to give $E$ the
  structure of a metric simplicial complex called the \emph{Coxeter
    complex of $W$}.  See Figure~\ref{fig:g2-axis}.  The
  top-dimensional simplices are known as \emph{chambers} and the one
  used to define the \emph{simple system} $S$ is the \emph{fundamental
    chamber}.  Other simple systems are obtained from other chambers.
  The resulting action of $W$ on its Coxeter complex preserves the
  simplicial structure and is uniquely transitive on chambers.  Using
  this action, the facets of any chamber, and the reflections that fix
  them, can be identified with the vertices of $\Gamma$ in a canonical
  way.  Also note that any proper subset of $S$ is a collection of
  reflections whose hyperplanes intersect in a facet of $\sigma$ and
  thus their product is elliptic.
\end{defn}

There is an alternative encoding of the geometry of an irreducible
euclidean Coxeter complex into a finite collection of vectors called
roots.

\begin{defn}[Roots]
  Let $W = \cox(\wt X_n)$ be an irreducible euclidean Coxeter group.
  If $r$ is a reflection in $W$ and $\alpha$ is any vector in $V$
  whose span is the line $L=\mov(r)$, then $\alpha$ is called a
  \emph{root of $r$}.  There is a finite collection of vectors $\Phi =
  \Phi_{X_n}$ called a \emph{root system} that contains a pair of
  roots $\pm \alpha$ for each family of parallel hyperplanes defining
  reflections in $W$ and the length of $\alpha$ encodes the minimal
  distance between these equally spaced parallel hyperplanes.  The
  root system that encode the hyperplanes of the $\wt G_2$ Coxeter
  group (shown in Figure~\ref{fig:g2-axis}) is illustrated in
  Figure~\ref{fig:phi-g2}.  Note that longer roots correspond to
  hyperplanes with shorter distances between them.  Root length is
  encoded in the extended Dynkin diagram as follows.  When two
  vertices are connected by a single edge, the roots they represent
  are the same length, when they are connected by a double edge, one
  root is $\sqrt{2}$ times the length of the other and when they are
  connected by a triple edge, one root is $\sqrt{3}$ times the length
  of the other.  The longer root is indicated by superimposing an
  inequality sign.
\end{defn}

\begin{figure}
  \includegraphics[scale=.8]{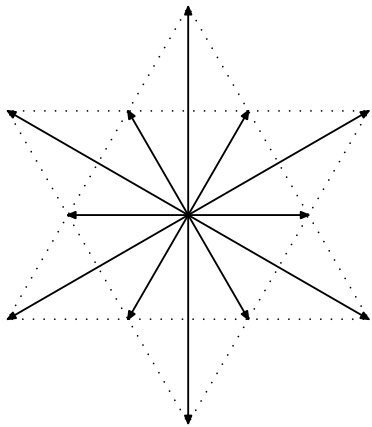}
  \caption{The root system $\Phi_{G_2}$.\label{fig:phi-g2}}
\end{figure}

The Coxeter complex of an irreducible euclidean Coxeter group can be
reconstructed from its root system because it always contains a point
$x$ with the property that every hyperplane of a reflection in $W$ is
parallel to a hyperplane of a reflection in $W$ fixing $x$.  In the
$\wt G_2$ example, $x$ can be any corner of a triangle with a $30$
degree angle.  After identifying $V$ and $E$ using this special point
as our origin, the reconstruction proceeds as follows.

\begin{defn}[Reflections and hyperplanes]\label{def:root-cox}
  Let $\Phi = \Phi_{X_n}$ be a root system of an irreducible euclidean
  Coxeter group $W = \cox(\wt X_n)$.  For each $\alpha \in \Phi$ and
  $i\in \Z$ let $H_{\alpha,i}$ denote the (affine) \emph{hyperplane}
  in $V$ of solutions to the equation $\langle x,\alpha \rangle = i$
  where the brackets denote the standard inner product on $V$.  The
  intersections of the hyperplanes give the simplicial structure.  The
  unique nontrivial isometry of $V$ that fixes $H_{\alpha,i}$
  pointwise is a reflection that we call $r_{\alpha,i}$.  The
  collection $R = \{r_{\alpha,i} \mid \alpha \in \Phi, i \in \Z\}$
  generates a euclidean Coxeter group $W$ and $R$ is its set of
  reflections in the sense of Definition~\ref{def:coxeter}.
\end{defn}

The reflections through the origin generate a finite Coxeter group
$W_0$ related to $W$ in two distinct ways.

\begin{defn}[Dynkin diagrams]\label{def:finite}
  The hyperplanes $H_\alpha = H_{\alpha,0}$ are precisely the ones
  that contain the origin and the reflections $r_\alpha =
  r_{\alpha,0}$ generate a finite Coxeter group $W_0$ that contains
  all elements of $W$ fixing the origin.  This embeds $W_0$ as a
  subgroup of $W$.  There is also a well-defined group homomorphism
  $p\colon W \onto W_0$ defined by sending each generating reflection
  $r_{\alpha,i}$ in $W$ to $r_\alpha$ in $W_0$.  Choosing a
  fundamental chamber containing the origin shows that $W_0$ is a
  Coxeter group generated by all but one of the reflections in $S$.
  The vertex of the extended Dynkin diagram $\Gamma$ corresponding to
  the missing reflection is shaded white.  When $\Gamma$ is an
  extended Dynkin diagram of type $\wt X_n$, the subgraph without the
  white vertex is called a \emph{Dynkin diagram of type $X_n$}.  In
  particular, the finite Coxeter group $W_0 = \cox(X_n)$.
\end{defn}

In the notation of Definition~\ref{def:root-cox}, the white dot
represents a reflection of the form $r_{\alpha,1}$ where $\alpha \in
\Phi$ is a canonical vector of ``highest weight''.  The translations
in $W$ are described by coroots and the coroot lattice.

\begin{defn}[Coroots]
  For each $\alpha \in \Phi$ consider the product $r_{\alpha,1}
  r_\alpha$, or equivalently $r_{\alpha,i+1} r_{\alpha,i}$.
  Reflecting through parallel hyperplanes produces a translation in
  the $\alpha$ direction and the exact translation is
  $t_{\alpha^\vee}$ where $\alpha^\vee = c \alpha$ is a \emph{coroot}
  with $c= \frac{2}{\langle \alpha,\alpha \rangle}$.  The collection
  of all coroots is denoted $\Phi^\vee$ and the integral linear
  combinations of vectors in $\Phi^\vee$ is a lattice $\Z(\Phi^\vee)
  \cong \Z^n$ called the \emph{coroot lattice}.  Because $t_\mu t_\nu
  = t_{\mu + \nu} = t_\nu t_\mu$, there is a translation of the form
  $t_\lambda$ in $W$ for each $\lambda \in \Z(\Phi^\vee)$ and the set
  $T = \{ t_\lambda \mid \lambda \in \Z(\Phi^\vee)\}$ forms an abelian
  subgroup of $W$.  In fact, these are the only translations that are
  contained in $W$ (i.e. $T = T_E \cap W$), the subgroup $T$ is the
  kernel of the map $p\colon W \onto W_0$ and $W$ is a semidirect
  product of $W_0$ and $T$.
\end{defn}

%%%%%%%%%%%%%%%%%%%%%%%%%%
\section{Coxeter elements}
%%%%%%%%%%%%%%%%%%%%%%%%%%

This section coarsely classifies Coxeter elements in irreducible
euclidean Coxeter groups.  Recall that a \emph{Coxeter element} in a
Coxeter group $W$ with standard generating set $S$ is a product of the
reflections of $S$ in some linear order.  We begin by determining the
basic geometric invariants of a Coxeter element in an irreducible
euclidean Coxeter group when viewed as a euclidean isometry.  The key
observation is that a collection of vectors normal to the facets of a
euclidean simplex are almost linearly independent in the sense that
every proper subset is linearly independent but the full set is not.

\begin{prop}[Simple systems and elliptic isometries] 
  \label{prop:sim-sys-iso} 
  If $S$ is a simple system of an irreducible euclidean Coxeter group
  $W$ corresponding to a chamber $\sigma$, then the product of any
  proper subset of the reflections in $S$ is an elliptic element whose
  fix-set is the affine hull of the face of $\sigma$ determined by the
  intersection of the corresponding hyperplanes.  Moreover, this is a
  minimum length reflection factorization of the resulting elliptic
  isometry.
\end{prop}

\begin{proof}
  The hyperplanes corresponding to any proper subset of $S$ have a
  face of $\sigma$ in common and the product of the corresponding
  reflections fixes its affine hull.  This shows that the product is
  elliptic.  The fact that this is the full fix-set and that this
  product of reflections has minimum length follows immediately from
  \cite[Lemma~$\LemMinEll$]{BrMc-factor} and the observation that the
  roots of these reflections are linearly independent.
\end{proof}

\begin{prop}[Coxeter elements are hyperbolic isometries]
  \label{prop:cox-elt-iso}
  A Coxeter element for an irreducible euclidean Coxeter group is a
  hyperbolic isometry of $E$, its move-set is a nonlinear affine
  hyperplane in $V$ and its min-set is a line in $E$.  Moreover, any
  factorization of this element as a product of the elements in a
  simple system is a minimum length reflection factorization.
\end{prop}

\begin{proof}
  Let $w \in W = \cox(\wt X_n)$ be the Coxeter element under
  discussion and let $w=r_0 r_1 \cdots r_n$ be a factorization of $w$
  as the product of the $n+1$ reflections in a simple system $S =
  \{r_0, \ldots, r_n\}$ corresponding to a chamber $\sigma$ in the
  Coxeter complex of $W$.  By Proposition~\ref{prop:sim-sys-iso}, the
  product $w_0 = r_1 \cdots r_n$ is an elliptic isometry that only
  fixes a single vertex of $\sigma$ and consequently its move-set is
  all of $V$ (Remark~\ref{rem:aux}).  The hyperplane of the reflection
  $r_0$ is the determined by the facet through the other $n$ vertices
  of $\sigma$ and in particular, it does not contain the point
  $\fix(w_0)$.  By \cite[Proposition~$\PropEllRefl$]{BrMc-factor}, the
  product $w = r_0 w_0$ is a hyperbolic element that has the listed
  properties.
\end{proof}

The line in $E$ that is the min-set of a Coxeter element $w$ in an
irreducible euclidean Coxeter group is called its \emph{axis}.  The
next step is to classify those Coxeter elements that are geometrically
distinct.  The first thing to note is that when standard generators
commute, distinct orderings can produce the exact same element.  In
fact, the only critical information is the ordering of pairs of
generators joined by an edge in the Coxeter diagram $\Gamma$.  If we
orient each edge of $\Gamma$ according to the order in which the
reflections corresponding to its endpoints occur in the fixed total
order, the result is an acyclic orientation of $\Gamma$ and it is easy
to prove that two linear orderings of $S$ that induce the same acyclic
orientation of $\Gamma$ produce the same element $w\in W$.  On the
hexagonal diagram for the $\wt A_5$ Coxeter group, for example, there
are $6! = 720$ different products of its $6$ standard generators but
at most $2^6-2 = 62$ distinct Coxeter elements produced since this is
the number of acyclic orientations.  There is also a coarser notion of
geometric equivalence.

\begin{defn}[Geometric equivalence]
  Call an automorphism $\psi:W \to W$ \emph{geometric} if $\psi$ sends
  reflections to reflections and simple systems to simple systems.
  When $W$ is an irreducible euclidean Coxeter group this is
  equivalent to being induced by a metric-preserving simplicial
  automorphism of the Coxeter complex.  The geometric automorphisms
  form a subgroup of the full automorphism group that contains the
  inner automorphisms and the automorphism induced by symmetries of
  the Coxeter diagram.  In fact, every geometric automorphism is a
  composition of an inner automorphism and a diagram automorphism.
  Call two elements $w$ and $w'$ \emph{geometrically equivalent} when
  there is a geometric automorphism $\psi$ sending $w$ to $w'$, the
  point being that geometrically equivalent elements have similar
  geometric properties.
\end{defn}

Geometric equivalence is sufficient for our purposes since
geometrically equivalent Coxeter elements produce intervals that are
identical after a systematic relabeling of the edges.  In particular,
geometrically equivalent Coxeter elements produce isomorphic dual
Artin groups.  To help identify Coxeter elements that are
geometrically equivalent, we use the following lemma with a
complicated statement and an easy proof.

\begin{lem}[Sources and sinks]\label{lem:sink-source}
  Let $W$ be a Coxeter group with fundamental chamber $\sigma$, let
  $w$ be the Coxeter element of $W$ produced by a fixed acyclic
  orientation of its diagram $\Gamma$ and let $r$ be a reflection with
  hyperplane $H$ associated with a vertex $v \in \Gamma$ that is
  either a source or a sink in this orientation.  Then the element $w$
  is also a product of the reflections associated with the chamber
  $r(\sigma)$ on the other side of $H$ with respect to the orientation
  of $\Gamma$ that agrees with the previous one except that the
  orientation is reversed for every edge incident with $v$.  In
  particular, two acyclic orientations of $\Gamma$ that differ by a
  single sink-source flip produce Coxeter elements that are
  geometrically equivalent.
\end{lem}

\begin{proof}
  Pick a linear ordering of the vertices consistent with the
  orientation of $\Gamma$ so that $w = r_0 r_1 r_2 \cdots r_n$ where
  the $r_i$ are the reflections through the facets of the chamber
  $\sigma$ and $r=r_0$ or $r=r_n$.  If $r=r_0$ then the assertion is a
  consequence of the elementary observation that $w = r r_1 r_2 \cdots
  r_n = r_1^r r_2^r \cdots r_n^r r$ where $a^b$ is shorthand for
  $bab^{-1}$.  The reflections in the second factorization bound the
  simplex $r(\sigma)$ which shares a facet with $\sigma$ and is the
  reflection of $\sigma$ across $H$ and the orientation of $\Gamma$
  induced by this factorization satisfies the given description.  The
  case $r=r_n$ is similar.
\end{proof}

It quickly follows that most irreducible euclidean Coxeter groups have
only one Coxeter element up to geometric equivalence.

\begin{prop}[Geometrically equivalent]\label{prop:geo-equiv}
  If $W$ is a Coxeter group whose diagram is a tree, then all of its
  Coxeter elements are geometrically equivalent. This holds, in
  particular, for every irreducible euclidean Coxeter group that is
  not type $A$.
\end{prop}

\begin{proof}
  There is an easy induction argument using
  Lemma~\ref{lem:sink-source} which proves that any two acyclic
  orientations of a tree are geometrically equivalent.  The rough idea
  is to remove a valence~$1$ vertex and the unique edge connected to
  it, apply the inductive hypothesis to this pruned tree and then use
  the resulting sequence of flips as a template for the original
  situation inserting flips of the removed valence~$1$ vertex as
  necessary in order to make sure the vertex at the other end of its
  unique edge is a sink/source when required.
\end{proof}

When the Coxeter diagram is a tree, there are exactly two orientations
under which every vertex is a source or a sink.  The two Coxeter
elements that result are inverses of each other and either one is
called a \emph{bipartite Coxeter element}.  Because of the way in
which Lemma~\ref{lem:sink-source} is proved, it is an immediate
consequence of Proposition~\ref{prop:geo-equiv} that every Coxeter
element of an irreducible euclidean Coxeter group that is not of type
$A$ can be viewed as a bipartite Coxeter element so long as the
fundamental chamber is chosen carefully.  The chambers which produce
$w$ as a bipartite Coxeter element have an elegant geometric
characterization that is described in the next section.

\begin{cor}[Bipartite Coxeter elements]\label{cor:bip-cox}
  If $w$ is a Coxeter element of an irreducible euclidean Coxeter
  group $W$ that is not of type $A$, then for each acyclic orientation
  of $\Gamma$ there exists a chamber $\sigma$ in its Coxeter complex
  which has $w$ as the Coxeter element determined by this orientation.
  In particular, there is a chamber which produces $w$ as its
  bipartite Coxeter element.
\end{cor}

The analysis of Coxeter elements in the group $W=\cox(\wt A_n)$ up to
geometric equivalent is slightly more delicate.  

\begin{defn}[Bigon Coxeter elements]\label{def:bigon}
  When $n$ is at least $2$, the diagram $\wt A_n$ is a cycle and the
  two edges adjacent to a sink or a source point in opposite
  directions around the cycle. (The case $n=1$ is covered by
  Proposition~\ref{prop:geo-equiv}.)  In particular, flipping sinks
  and sources does not change the number of edges pointing in each
  direction.  Moreover, using flips and diagram symmetries it is clear
  that any Coxeter element produced by an acyclic orientation is
  geometrically equivalent to one produced by an acyclic orientation
  in which there is a unique sink, a unique source, $p$ consecutive
  edges pointing in the clockwise direction and $q$ consecutive edges
  pointing in the counterclockwise direction with $p \geq q$ and $p+q
  = n+1$.  We call a Coxeter element $w$ derived from such an
  orientation a \emph{$(p,q)$-bigon Coxeter element}.  We should also
  note that in many respects a Coxeter element of $\wt A_1$ can be
  considered a $(1,1)$-bigon Coxeter element.
\end{defn}

In the Coxeter group of type $\wt A_5$, the $62$ acyclic orientations
of its Coxeter diagram describe at most $3$ geometrically distinct
Coxeter elements since each is one is geometrically equivalent to a
$(p,q)$-bigon Coxeter element where $(p,q)$ is either $(5,1)$, $(4,2)$
or $(3,3)$.  More generally, the Coxeter group of type $\wt A_n$ has
at most $\frac{n+1}{2}$ geometrically distinct Coxeter elements since
this is an upper bound on $q$.  The following is the type $A$ analog
of Corollary~\ref{cor:bip-cox}.

\begin{cor}[Bigon Coxeter elements]\label{cor:bigon-cox}
  If $w$ is a Coxeter element of an irreducible euclidean Coxeter
  group of type $\wt A_n$, then there is a chamber $\sigma$ in its
  Coxeter complex which produces $w$ as one of its bigon Coxeter
  elements.
\end{cor}

The upshot of this analysis is that there is exactly one dual Artin
group up to isomorphism for each irreducible euclidean Artin group
that is not of type $A$ and when $W = \cox(\wt A_n)$ there are at most
$\frac{n+1}{2}$ such dual groups.

%%%%%%%%%%%%%%%%%%%%%%%%%%%%%%%%%%%%
\section{Bipartite Coxeter elements}
%%%%%%%%%%%%%%%%%%%%%%%%%%%%%%%%%%%%

The next two sections are a slight digression into the geometry of
bipartite Coxeter elements in irreducible euclidean Coxeter groups.
They are not needed to prove Theorem~\ref{main:dual} but this is a
convenient location to establish various results for use in the next
article in the series \cite{McSu-artin-euclid}.  Let $W$ denote an
irreducible euclidean Coxeter group that is not of type $A$, let $w$
be one of its Coxeter elements and let $\sigma$ be a chamber in its
Coxeter complex.  The goal in this section is to establish a close
geometric relationship between the axis of $w$ and the chambers
$\sigma$ that produce $w$ as a bipartite Coxeter element.  More
precisely, we show that these chambers are exactly those whose
interior intersects the axis of $w$ (Theorem~\ref{thm:pierced}).  We
begin by focusing on the geometry of $\sigma$.

\begin{defn}[Bipartite faces and subspaces]
  Since $W$ is not of type $A$, its Coxeter diagram $\Gamma$ is a tree
  with a unique bipartite structure.  For any chamber $\sigma$ in the
  Coxeter complex of $W$ this leads to a pair of distinguished
  disjoint faces in $\sigma$.  More explicitly, let $S_0 \sqcup S_1 =
  S$ be the bipartite partitioning of the reflections determined by
  the facets of $\sigma$ corresponding to the unique bipartite
  structure on $\Gamma$, let $F_i$ be the face of $\sigma$ determined
  by the intersection of the hyperplanes of the reflections in $S_i$
  and let $B_i$ be the affine hull of $F_i$.  Note that since each
  hyperplane is determined by the vertex of $\sigma$ that it does not
  contain, the face $F_0$ is the convex of the vertices not contained
  in the various reflections in $S_1$ and the $F_1$ is the convex hull
  of the vertices not contained in the various reflections in $S_0$.
  The affine subspaces $B_0$ and $B_1$ are disjoint since the
  hyperplanes determined by the facets of $\sigma$ have trivial
  intersection.  We call $F_0$ and $F_1$ the \emph{bipartite faces} of
  $\sigma$ and $B_0$ and $B_1$ the \emph{bipartite subspaces} of
  $\sigma$.
\end{defn}

In the $\wt G_2$ example $F_0$ and $F_1$ are a point and the
hypotenuse of the 30-60-90 right triangle, respectively, and $B_0$ and
$B_1$ are the point and line they determine.  Before continuing we
pause to record some elementary observations about euclidean
simplices.

\begin{rem}[Euclidean simplices]\label{rem:euc-simp}
  Consider the general situation where $\sigma$ is a euclidean
  $n$-simplex embedded in a euclidean space equal to its affine hull
  and $F_0$ and $F_1$ are disjoint faces of $\sigma$ that collectively
  contain all of its vertices (conditions satisfied by the bipartite
  faces defined above).  If $B_i$ is the affine hull of $F_i$, then
  $B_0$ and $B_1$ are disjoint, the linear subspaces $\dir(B_0)$ and
  $\dir(B_1)$ have trivial intersection and the subspace spanned by
  their union has codimension~$1$ in $V$.  In particular, $(\dir(B_0)
  \cup \dir(B_1))^\perp = \dir(B_0)^\perp \cap \dir(B_1)^\perp$ is a
  line $L$.  Next, there exist a unique pair of distinct points $x_i
  \in B_i$ that realize the minimal distance between $B_0$ and $B_1$.
  The uniqueness of $x_0$ and $x_1$ follows from the properties of
  $\dir(B_0)$ and $\dir(B_1)$ mentioned above.  Since the line segment
  connecting these distance minimizing points is necessarily in a
  direction orthogonal to both $\dir(B_0)$ and $\dir(B_1)$, its
  direction vector spans the line $L$.
\end{rem}

\begin{defn}[Bipartite lines and closest points]\label{def:bip-line}
  Let $B_0$ and $B_1$ be the bipartite subspaces of a chamber $\sigma$
  in the Coxeter complex of an irreducible euclidean Coxeter group
  that is not of type $A$.  The unique pair of points $x_i \in B_i$
  that realize the minimum distance between $B_0$ and $B_1$ are called
  \emph{closest points} and the unique line they determine is the
  \emph{bipartite line} of $\sigma$.
\end{defn}

In $W = \cox(\wt G_2)$, the points $x_0$ and $x_1$ are the vertex with
the right angle and the foot of the altitude dropped to the
hypotenuse.  In Figure~\ref{fig:g2-axis} the dashed line is the
bipartite line for each of the heavily shaded triangles through which
it passes.  In general, there is a practical method for finding the
direction of the bipartite line which is particularly useful once the
Coxeter complex has more dimensions than can be easily visualized.

\begin{rem}[Direction of the bipartite line]\label{rem:axis-dir}
  Let $\sigma$ be a chamber in the Coxeter complex of an irreducible
  euclidean Coxeter group that is not of type $A$.  The roots of the
  $n+1$ reflections in the corresponding simple system $S$ are
  linearly dependent and this essentially unique linear dependency
  must necessarily involve all $n+1$ roots since any proper subset is
  linear independent.  If we separate the terms of the equation
  according to the bipartite subdivision $S = S_0 \sqcup S_1$ so that
  the roots corresponding to the reflections in $S_0$ are on the left
  hand side and the roots corresponding to the reflections in $S_1$
  are on the right hand side, then the vector $\lambda$ described by
  either side of this equation is the direction of the bipartite line.
  To see this note that by construction $\lambda$ is nontrivial
  (because of the linear independent of proper subsets of the roots)
  and it can be written as a linear combination of either the $S_0$
  roots or the $S_1$ roots.  In particular, $\lambda$ is in
  $\dir(B_0)^\perp \cap \dir(B_1)^\perp = L$, where $L$ is the
  direction of the bipartite line of $\sigma$.  This procedure is used
  in Section~\ref{sec:computation}.
\end{rem}

Returning to the $\wt G_2$ example, notice that $x_i$ lies in the
interior of face $F_i$ rather than elsewhere in its affine hull $B_i$.
This is, in fact, always the case.

\begin{lem}[Interior]\label{lem:interior}
  When $\sigma$, $F_i$, $B_i$ and $x_i$ are defined as above, the
  closest point $x_i \in B_i$ lies in the interior of the face $F_i$.
  In particular, the bipartite line of $\sigma$ intersects the
  interior of $\sigma$.
\end{lem}

\begin{proof}
  Let $y_i \in F_i$ be points that realize the minimum distance
  between the faces $F_0$ and $F_1$.  We first show that $y_0$ lies in
  the interior of $F_0$.  The key facts are that every dihedral angle
  in $\sigma$ is non-obtuse and that the defining diagram $\Gamma$ is
  connected.  In particular, if $y_0$ is in the boundary of $F_0$ then
  there there exist hyperplanes $H_0$ and $H_1$ determined by facets
  with $H_i \supset F_i$ and $y_0 \in H_1 \cap F_0$ such that the
  dihedral angle between $H_0$ and $H_1$ is acute.  This means that
  the distance can be shrunk by moving $y_0$ into the interior of a
  higher dimensional face of $F_0$, contradiction.  Thus $y_0$ is in
  the interior of $F_0$.  After reversing the roles of $0$ and $1$, we
  see that $y_1$ is in the interior of $F_1$.  This means that the
  vector from $y_0$ to $y_1$ is orthogonal to both affine spans and is
  in the direction of the line $L$.  In particular, the points $x_0$,
  $y_0$, $y_1$ and $x_1$ form a possibly degenerate rectangle where
  $x_0$ and $x_1$ are the unique points realizing the minimum distance
  between $B_0$ and $B_1$.  But because $\dir(B_0)$ and $\dir(B_1)$
  have no nontrivial vector in common, $x_0 = y_0$ and $x_1 = y_1$.
\end{proof}

The next step is to establish that the bipartite line of a chamber
is the min-set of its bipartite Coxeter elements.

\begin{defn}[Bipartite involutions]\label{def:bip-inv}
  Let $\sigma$ be a chamber in a Coxeter complex of $W$, let $S$ be
  the corresponding simple system, and let $S = S_0 \sqcup S_1$ be its
  bipartite decomposition.  Because the reflections in $S_i$ pairwise
  commute, the product $w_i$ of the reflections in $S_i$ is
  independent of the order in which they are multiplied and it is an
  involution.  We call $w_0$ and $w_1$ the \emph{bipartite
    involutions} of $\sigma$ and note that the two bipartite Coxeter
  elements of $\sigma$ are $w = w_1 w_0$ and $w^{-1} = w_0 w_1$.
  Geometrically, $w_i$ fixes $B_i$ pointwise and (if we pick a point
  in $B_i$ as the origin) it acts as the antipodal map on its
  orthogonal complement.
\end{defn}

In the $\wt G_2$ example, $w_0$ is a $180^\circ$ rotation about $x_0$
and $w_1$ is a reflection fixing the horizontal line $B_1$.  The
description of the action of $w_i$ given above establishes the
following lemma from which we conclude that the bipartite line of a
chamber is the axis of its bipartite Coxeter elements.

\begin{lem}[Reflecting the bipartite line]\label{lem:refl-bip-line}
  If $\sigma$ is a chamber in the Coxeter complex of $W$ with closest
  points $x_i$ and bipartite line $L$, then its bipartite involution
  $w_i$ restricts to a reflection on $L$ fixing only $x_i$.
\end{lem}

\begin{prop}[Bipartite lines as axes]\label{prop:bip-axes}
  The bipartite line of any chamber $\sigma$ in the Coxeter complex of
  $W$ is the axis of the bipartite Coxeter elements produced by
  $\sigma$.
\end{prop}

\begin{proof}
  Let $w_0$ and $w_1$ be the bipartite involutions of $\sigma$.  By
  Lemma~\ref{lem:refl-bip-line} the product $w = w_1w_0$ stabilizes
  $L$ and acts as a translation on $L$.  By the characterization of
  min-sets quoted in Remark~\ref{rem:aux}, $L$ is the axis of $w$.
  The same reasoning show that $L$ is the axis of $w^{-1} = w_0 w_1$.
\end{proof}

The bipartite involutions of $\sigma$ can be used to extend our
notation.

\begin{defn}[Axial chambers]\label{def:axial-notation}
  Let $\sigma$ be a chamber in the Coxeter complex of $W$ with
  bipartite faces $F_0$ and $F_1$, bipartite subspaces $B_0$ and
  $B_1$, and closest points $x_0$ and $x_1$.  The bipartite
  involutions $w_0$ and $w_1$ define an infinite dihedral group action
  on the line $L$ through $x_0$ and $x_1$.  Using this action, we can
  extend the definitions of $F_i$, $B_i$, $x_i$ and $\sigma_i$ to
  arbitrary subscripts $i \in \Z$ by letting $F_{-i}$ / $B_{-i}$ /
  $x_{-i}$ denote the image of $F_i$ / $B_i$ / $x_i$ under $w_0$ and
  letting $F_{2-i}$ / $B_{2-i}$ / $x_{2-i}$ denote the image of $F_i$
  / $B_i$ / $x_i$ under $w_1$.  The result is a sequence of equally
  spaced points $x_i$ that occur in order along $L$, one for each
  $i\in \Z$.  Finally, let $\sigma_i$ denote the image of $\sigma$
  under this dihedral group action that contains $x_i$ and $x_{i+1}$
  (so that $\sigma = \sigma_0$).  We call these chambers \emph{axial
    chambers} and their vertices are \emph{axial vertices}.
\end{defn}

The axial chambers in Figure~\ref{fig:g2-axis} are the ones that are
heavily shaded.  One fact about the arrangement of the chambers
$\sigma_i$ along the axis $L$ that is important to note is that every
point of $L$ that is not one of the points $x_i$ lies in the interior
of some chamber $\sigma_i$.

\begin{thm}[Axial chambers and Coxeter elements]\label{thm:pierced}
  Let $L$ be the axis of a Coxeter element $w$ for an irreducible
  euclidean Coxeter group $W$ that is not of type $A$ and let $\sigma$
  be a chamber in its Coxeter complex.  The chamber $\sigma$ produces
  $w$ as a bipartite Coxeter element iff the line $L$ intersects the
  interior of $\sigma$.
\end{thm}

\begin{proof}
  When $\sigma$ produces $w$ as a bipartite Coxeter element then by
  Proposition~\ref{prop:bip-axes} the axis of $w$ is the bipartite
  line of $\sigma$ and by Lemma~\ref{lem:interior} this line
  intersects the interior of $\sigma$.  In the other direction, we
  know that there is at least one chamber $\sigma$ that produces $w$
  as a bipartite Coxeter element and this $\sigma$ contains a portion
  of $L$ in its interior.  It is sufficient to note that interiors of
  chambers are disjoint open subsets of $A$ and that the interiors of
  the axial chambers under the infinite dihedral group action
  generated by the bipartite involutions of $\sigma$ contain all of
  $L$ except the discrete set of points $x_i$ with $i\in \Z$. 
\end{proof}

Another way to phrase this result is that a chamber $\sigma$ produces
$w$ as a bipartite Coxeter element iff $\sigma$ is a chamber in the
smallest simplicial subcomplex containing the axis of $w$.  We
conclude this section with one final observation.

\begin{cor}[Hyperplanes crossing the axis]\label{cor:hyper-axis}
  Let $L$ be the axis of a Coxeter element $w$ in an irreducible
  euclidean Coxeter group $W$ that is not of type $A$.  If a
  hyperplane $H$ of a reflection in $W$ crosses the line $L$ then $H$
  is determined by a facet of an axial simplex.  More precisely, there
  is an index $i$ such that $H$ contains all of $F_i$, all but one
  vertex of $F_{i-1}$ and all but one vertex of $F_{i+1}$.
\end{cor}

\begin{proof}
  The intersection of $H$ and $L$ must occur at one of the points
  $x_i$ since the remainder of $L$ is covered by the interiors of
  axial simplices as noted after Definition~\ref{def:axial-notation}.
  Moreover, since $x_i$ lies in the interior of the face $F_i$, $H$
  contains all of $F_i$ and thus all of its affine hull $B_i$.  Next,
  recall that the facets of $\sigma_i$ containing $F_i$ determine
  hyperplanes representing pairwise commuting reflections, and as a
  result they intersect the orthogonal complement of $B_i$ (based at
  $x_i$) in an arrangement that looks like the standard coordinate
  hyperplanes with the link of $B_i$ in $\sigma_i$ forming one of its
  orthants.  As this orthant is a chamber of the finite reflection
  subgroup fixing $B_i$, it is clear that these are the only
  hyperplanes of $W$ that contain $B_i$.  In particular, $H$ must
  itself be a hyperplane determined by a facet of $\sigma_i$
  containing $F_i$ and thus have the listed properties.
\end{proof}

%%%%%%%%%%%%%%%%%%%%%
\section{Reflections}
%%%%%%%%%%%%%%%%%%%%%

Let $w$ be a Coxeter element in an irreducible euclidean Coxeter group
$W$ that is not of type $A$.  In this section we determine the set
$R_0$ of reflections that occur in some minimal length reflection
factorization of $w$.  We call these the reflections below $w$.

\begin{defn}[Reflections below $w$]
  It follows easily from Lemma~\ref{lem:rewriting} that for any
  reflection $r$ in $W$ the following conditions are equivalent: (1)
  $\lr(rw) < \lr(w)$ (2) $r$ is the leftmost reflection in some
  minimal length factorization of $w$ (3) $r$ is a reflection in some
  minimal length factorization of $w$ (4) $r$ is the rightmost
  reflection in some minimal length factorization of $w$ and (5)
  $\lr(wr) < \lr(w)$.  When these conditions hold, we say that $r$ is
  a \emph{reflection below $w$}.
\end{defn}

When analysizing the reflections below a Coxeter element in an
irreducible euclidean Coxeter group it is useful to distinguish two
types of reflections.

\begin{defn}[Vertical and horizontal reflections]\label{def:v-h}
  The Coxeter axis is a line in the euclidean space $E$ and its space
  of directions is a line in the vector space $V$ that has a
  hyperplane as its orthogonal complement.  We call the vectors in
  this hyperplane \emph{horizontal} and those in this line
  \emph{vertical}.  More generally, any vector with a nontrivial
  vertical component (i.e. any vector not in the hyperplane) is also
  called vertical.  Using this distinction, we separate the
  reflections below $w$ into two types based on the type of its roots.
  In other words, a reflection $r$ is \emph{horizontal} if its root is
  orthogonal to the direction of the Coxeter axis and \emph{vertical}
  otherwise.
\end{defn}

The main reason to distinquish vertical and horizontal reflections is
that when a Coxeter element of an irreducible euclidean Coxeter group
is multiplied by a vertical reflection, the result is elliptic and
when it is multiplied by a horizontal reflection, the result is
hyperbolic \cite{BrMc-factor}.

\begin{lem}[Vertical reflections]\label{lem:vert-refl}
  Let $L$ be the axis of a Coxeter element $w$ in an irreducible
  euclidean Coxeter group $W$ that is not of type $A$ and let $H$ be
  the hyperplane of a vertical reflection in $W$ that intersects $L$
  at the point $x_i$.  If $u$ and $v$ are the unique vertices of
  $F_{i-1}$ and $F_{i+1}$ not contained in $H$, then $w$ sends $u$ to
  $v$, $r$ swaps $u$ and $v$, $rw$ fixes $u$, and $wr$ fixes $v$.
  Moreover, the elliptic isometry $rw$ is a Coxeter element for the
  finite Coxeter subgroup of $W$ that stabilizes $u$ and the elliptic
  isometry $wr$ is a Coxeter element for the finite Coxeter subgroup
  of $W$ that stabilizes $v$.
\end{lem}

\begin{proof}
  Recall that $w_i$ fixes $F_i$ and acts as the antipodal map on its
  orthogonal complement sending $\sigma_{i-1}$ to $\sigma_i$.  In
  particular, it stablizes any hyperplane that contains $F_i$ and thus
  sends the vertices in $H \cap \sigma_{i-1}$ to the vertices in $H
  \cap \sigma_i$.  This means that $w_i$ must send the one remaining
  vertex $u$ to the one remaining vertex $v$.  Because $w = w_i
  w_{i-1}$ and $w_{i-1}$ fixes all of $F_{i-1}$, we have that $w(u) =
  w_i w_{i-1} (u) = w_i (u) = v$.  Next, because $r$ is one of the
  commuting reflections whose product is $w_i$, $r w_i$ is the same
  product with $r$ deleted and $rw = r w_i w_{i-1}$ is the product of
  all the reflections defined by facets of $\sigma_{i-1}$ except the
  reflection defined by the facet whose affine span is the hyperplane
  of $r$.  In particular, all of these reflections contain the vertex
  $u$ in their fixed hyperplanes, their product fixes $u$ and they
  bound a spherical simplex formed by intersecting $\sigma_{i-1}$ with
  a small sphere centered at $u$.  Since these reflections bound a
  chamber of the corresponding spherical Coxeter complex, their
  product $rw$ is a Coxeter element for this subgroup.  Similarly, the
  factorization $w = w_{i+1} w_i$ shows that $wr$ is a product of the
  reflections determined by the facets of $\sigma_i$ with $r$ removed
  and this product is a Coxeter element for the stabilizer of $v$.
  Finally $r$ swaps $u$ and $v$ since $w$ sends $u$ to $v$, $rw$ fixes
  $u$, and $r$ has order $2$.
\end{proof}

From Lemma~\ref{lem:vert-refl} the following is immediate.

\begin{prop}[Vertical reflections]\label{prop:vert-refl}
  Let $w$ be a Coxeter element in an irreducible euclidean Coxeter
  group $W$ that is not of type $A$.  If $r$ is a reflection that is
  vertical with respect to axis of $w$ then $r$ contains many axial
  vertices in its fixed hyperplane, it is part of a bipartite
  factorization of $w$ and it is contained in the set $R_0$.
\end{prop}

For horizontal reflections, a more precise statement is necessary.

\begin{prop}[Horizontal reflections]\label{prop:hort-refl}
  Let $w$ be a Coxeter element in an irreducible euclidean Coxeter
  group $W$ that is not of type $A$.  If $r$ is a reflection that is
  horizontal with respect to the axis of $w$ then $r$ in $W$ is
  contained in a minimal length reflection factorization of $w$ and
  thus in $R_0$ iff the hyperplane of $r$ contains at least one axial
  vertex of $w$.
\end{prop}

\begin{proof}
  If $r$ is contained in $R_0$ then there is a factorization
  $r_0r_1\cdots r_n = w$ containing $r$ and by
  Lemma~\ref{lem:rewriting} we can assume that $r=r_n$.  Since every
  point under $w$ is moved in the vertical direction, at least one of
  the reflections in this factorizations must be vertical and by
  Lemma~\ref{lem:rewriting} we can move this reflection to the $r_0$
  position without altering $r=r_n$.  This shows that $r$ is a
  reflection below $w' = r_0 w = r_1r_2 \cdots r_n$ which is an
  elliptic isometry fixing an axial vertex $v$
  (Proposition~\ref{prop:vert-refl}).  But this implies that $\fix(r)
  = H$ must contain $\fix(w') = v$ since $\fix(w')$ is the
  intersection of the hyperplanes of this minimal length reflection
  factorization of $w'$ \cite[Lemma~$\LemMinEll$]{BrMc-factor}.  In
  particular, $H$ contains an axial vertex.

  In the other direction, let $H$ be the hyperplane of $r$, let $u$ be
  an axial vertex contained in $H$, and using the notation of
  Definition~\ref{def:axial-notation}, let $F_{i-1}$ be the face
  containing $u$.  If we let $r'$ be the reflection defined by the
  facet of $\sigma_{i-1}$ not containing $u$ then by
  Lemma~\ref{lem:vert-refl} $r'w$ is a Coxeter element of the finite
  Coxeter group that stabilizes $u$, a group that contains $r$.  Since
  it is well-known that every reflection in a finite Coxeter group
  occurs in some minimal length factorization of any of its Coxeter
  elements, $r$ occurs in such a factorization $r_1 r_2 \cdots r_n$ of
  $r'w$.  The product $r' r_1 r_2 \cdots r_n$ is then a minimal length
  reflection factorization of $w$ that contains $r$ and $r$ belongs to
  $R_0$.
\end{proof}

Propositions~\ref{prop:vert-refl} and~\ref{prop:hort-refl} immediately
establish the following.

\begin{thm}[Reflections]\label{thm:refl}
  Let $w$ be a Coxeter element for an irreducible euclidean Coxeter
  group $W$ that is not of type $A$ and let $R_0$ be the set of
  reflections below $w$.  A reflection $r$ is in $R_0$ iff the
  hyperplane $H = \fix(r)$ contains an axial vertex.
\end{thm}

Although it requires a separate argument, we should note that
Theorem~\ref{thm:refl} also holds when $W$ is an irreducible euclidean
Coxeter group of type $\wt A_n$ and $w$ is any of its geometrically
distinct Coxeter elements.  Also, although every elliptic element in
the interval $[1,w]$ has as its fixed-set an affine subspace of the
Coxeter complex that contains at least one axial vertex, not all such
subspaces occur.  This should not be too surprising since this is
similar to the situation in finite Coxeter groups where only certain
``noncrossing'' subspaces occur as fixed sets of elements below the
Coxeter element.

Another important remark is that the length of an element with respect
to the full set of all reflections $R_E$ is, in general, quite
different from its length with respect to the set of reflections in
$W$.  See \cite{McCammondPetersen-refl-bound} where this is discussed
in detail.  This distinction does not play too large of a role in the
current context because the two length functions agree for a Coxeter
element $w$ and thus they agree for all of the elements in the
interval $[1,w]$.

This section concludes with a discussion of the dual presentation of
$\art(\wt G_2)$ that is designed to make Theorem~\ref{main:dual} more
comprehensive.  Craig Squier successfully analyzed its group structure
in \cite{Squier87} but he did so prior to the development of the
theory of Garside groups.  The following theorem shows that its dual
presentation is a Garside presentation in the sense of Digne.

\begin{thm}[Type $G$]\label{thm:g2}
  The interval used to define the unique dual presentation of the
  irreducible euclidean Artin group $\art(\wt G_2)$ is a lattice and
  thus its dual presentation is a Garside presentation.
\end{thm}

\begin{proof}
  Let $w$ be a Coxeter element of $W = \cox(\wt G_2)$.  Since it has a
  unique Coxeter element of $W$ up to geometric equivalence, its axis
  and axial simplices can be arranged as in Figure~\ref{fig:g2-axis}
  where the axis is dashed and the axial simplices are heavily shaded.
  By Theorem~\ref{thm:refl} the reflections that occur in minimal
  length factorizations of this glide reflection are the vertical
  reflections whose fixed lines cross the axis and the two horizontal
  reflections fixing one of the two vertical lines bounding the
  lightly shaded region in Figure~\ref{fig:g2-axis}.  (Multiplying $w$
  by a reflection $r$ fixing one of the other vertical lines results
  in a pure translation which has length $2$ with respect to $R_E$ but
  which does not have length $2$ with respect to $R$ since it
  translates in a direction that is not one of the root directions.)
  In this situation we can list the basic invariants of all of the
  isometries in the interval $[1,w]$.  The only hyperbolic isometries
  strictly below $w$ are the two translations that result when $w$ is
  multiplied by one of the two horizontal reflections.  Each
  translation is in a direction at a $30^\circ$ angle with the
  horizontal, one to left and one to the right.  The only length $2$
  factorizations of these translations are, of course, obtained by
  multiplying two parallel reflections whose root is in this
  direction.  In addition there is one elliptic isometry below $w$ for
  each axial vertex, one for each visible line in the Coxeter complex
  through an axial vertex and one for the entire plane.  With these
  descriptions it is relatively straightforward to check that the
  poset is a lattice.  Because the rank of the poset is so small any
  bowtie in the interval would be between two elements of rank $1$ and
  two elements of rank $2$.  Consider two elements at rank $2$,
  i.e. two elements that are either translations or rotations fixing a
  point.  The unique meet of the two translation strictly below $w$ is
  the identity since they have no common reflections in their possible
  factorizations.  The unique meet of a translation and a rotation is
  either the reflection through the line perpendicular to the
  translation direction that contains the fixed point of the rotation
  if such a line exists, or the identity otherwise.  And finally the
  meet of two rotations below $w$ is the reflection that fixes the
  line through their fixed points, if it exists in the Coxeter
  complex, or the identity otherwise. Finally, since well-defined
  meets always exists between any two elements of rank $2$ there are
  no bowties.  As remarked in Definition~\ref{def:bowtie} this means
  the interval $[1,w]$ is a lattice.
\end{proof}

%%%%%%%%%%%%%%%%%
\section{Bowties}
%%%%%%%%%%%%%%%%%

In this section we establish a criterion which implies that a Coxeter
interval in an irreducible euclidean Coxeter group is not a lattice.
This turns out to be the key result needed to establish
Theorem~\ref{main:dual}.  We begin with an elementary example in the
euclidean plane.

\begin{figure}
  \includegraphics{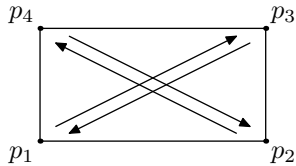}
  \caption{The rectangle with labeled corners used in
    Example~\ref{ex:rect}.\label{fig:rectangle}}
\end{figure}

\begin{exmp}[Reflections and translations]\label{ex:rect}
  Fix a rectangle in the euclidean plane with horizontal and vertical
  sides and label its corners as in Figure~\ref{fig:rectangle}.  Next,
  consider the group $G$ generated by the following eight elements:
  the four reflections that fix one of the four sides and the four
  translations that send a corner of this rectangle to the opposite
  corner.  The reflection fixing $p_i$ and $p_j$ is denoted $r_{ij}$
  and the translation sending $p_i$ to $p_j$ is $t_{ij}$.  Note that
  the subscripts of $r_{ij}$ are unordered and the subscripts of
  $t_{ij}$ are ordered.  In this notation $r_{12}$ and $r_{34}$ are
  reflections fixing horizontal lines, $r_{23}$ and $r_{41}$ are
  reflections fixing vertical lines and the four translations are
  $t_{13}$, $t_{31}$, $t_{24}$ and $t_{42}$.  The group $G$ contains
  the isometry $w$ that rotates the rectangle $180^\circ$ about its
  center.  For example, $w$ can be factored as $w = t_{13} r_{12}
  r_{41}$.  That this factorization is as short as possible over this
  generating set follows from the fact that every factorization must
  contain at least one horizontal reflection, one vertical reflection
  and one translation.  In fact, it is straightforward to check that
  there are exactly $24$ such factorizations of length~$3$.  The
  interval $[1,w]$ is shown in Figure~\ref{fig:rect-fact}.  The dashed
  lines represent translations and the solid lines represent
  reflections with the thickness distinguishing horizontal and
  vertical reflections.  Two final notes.  The interval $[1,w]$ is not
  a lattice because there is a bowtie connecting the two leftmost
  vertices in each of the two middle rows.  Concretely, the four
  factorizations of $w$ involved are $r_{34} t_{13} r_{41}$, $r_{34}
  t_{24} r_{23}$, $r_{12} t_{31} r_{23}$ and $r_{12} t_{42} r_{41}$.
  And finally, these $24$ factorizations of $w$ form a single closed
  orbit under the Hurwitz action.
\end{exmp}

\begin{figure}
  \includegraphics[scale=.8]{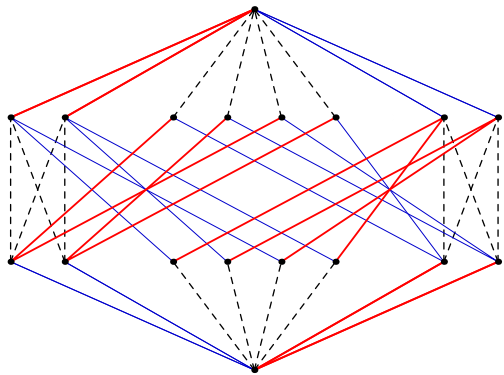}
  \caption{The poset of minimal factorizations described in
    Example~\ref{ex:rect}.\label{fig:rect-fact}}
\end{figure}

The $2$-dimensional configuration described in Example~\ref{ex:rect}
captures the essential reason why many of the intervals that define
dual euclidean Artin groups fail to be lattices.

\begin{prop}[Bowties]\label{prop:bowtie}
  Let $r$ and $r'$ be orthogonal reflections in an irreducible
  euclidean Coxeter group $W$ and let $t = t_\lambda$ be a translation
  in $W$ in a root direction.  If $t$ does not commute with either
  reflection and $\lambda$ is not in the plane spanned by the roots of
  $r$ and $r'$, then the interval $[1,w]$ below the element $w = t r
  r'$ contains a bowtie and is not a lattice.
\end{prop}
 
\begin{proof}
  The first step is to show that the interval $[1,w]$ contains a copy
  of Figure~\ref{fig:rect-fact} as an induced subposet.  Because
  $t_\lambda$ is a translation in a root direction it is a product of
  two parallel reflections in $W$.  This means that $w$ has reflection
  length at most $4$.  Next, note that $rr'$ is an elliptic element
  with a $2$-dimensional move-set $U$ spanned by the roots of $r$ and
  $r'$.  The hypothesis that $\lambda$ does not lie in $U$ means that
  the move-set of $t r r'$ is not through the origin and thus is a
  nonlinear $2$-dimensional affine subspace.  Thus $w$ is a hyperbolic
  isometry with reflection length at least $4$.  Combining these facts
  shows that $\lr(w)= 4$ and, after factoring $t$ into two parallel
  reflections in $W$, we have one of its minimal length factorizations
  and a corresponding maximal chain in $[1,w]$.

  Next, let $B$ be any $2$-dimensional affine subset of the euclidean
  space $E$ with $\dir(B) = U$ and note that the hyperplanes fixed by
  the reflections $r$ and $r'$ intersect $B$ in orthogonal lines
  $\ell$ and $\ell'$ and that both $r$ and $r'$ stabilize $B$.  If we
  uniquely decompose $\lambda$ as a vector $\lambda_1$ in $U$ plus a
  vector $\lambda_2$ in $U^\perp$ then we can factor $t$ into a pair
  of translations $t_{\lambda_1}$ and $t_{\lambda_2}$, one with a
  translation direction in $U$ and one with a translation direction in
  $U^\perp$.  The translation $t_{\lambda_2}$ commutes with $r$, $r'$
  and $t_{\lambda_1}$ and the hypotheses on $t$ ensure that
  $t_{\lambda_1}$ is nontrivial and not a direction vector of $\ell$
  or $\ell'$.  In particular, the way $r$, $r'$ and $t_{\lambda_1}$
  act on $U$ can be esssentially identified with the action of
  $r_{12}$, $r_{41}$ and $t_{13}$ of Example~\ref{ex:rect}.

  As we use the Hurwitz orbit in the example to alter the
  factorization there, we can mimic that action on the minimal length
  reflection factorizations of $w$, treating the translation as a
  product of two parallel reflections that always stay together and
  where both get conjugated simultaneously when necessary.  Under this
  action, the translation $t_{\lambda_2}$ simply follows the
  conjugates of the translation $t_{\lambda_1}$ around by virtue of
  the fact that it commutes with all three actions on $B$.  The result
  is a copy of Figure~\ref{fig:rect-fact} as an induced subposet of
  $[1,w]$.  The final step is to note that the bowtie visible on the
  left of Figure~\ref{fig:rect-fact} remains a bowtie in the larger
  poset $[1,w]$ because $[1,w]$ is, in turn, an induced subposet of
  the hyperbolic poset $P^{\mov(w)}$ where these same four elements,
  two of rank $1$ and two of rank $3$ are four elements forming one of
  the known bowties in $P^{\mov(w)}$ as described in
  Theorem~\ref{thm:hyp-bowtie}.
\end{proof} 

Translations and reflections satisfying the hypotheses of
Proposition~\ref{prop:bowtie} can be found below a Coxeter element in
an irreducible euclidean Coxeter group whenever the roots orthogonal
to the direction of its Coxeter axis form a reducible root system.

\begin{thm}[Reducibility and bowties]\label{thm:bowtie}
  Let $L$ be the axis of a Coxeter element $w$ in an irreducible
  euclidean Coxeter group $W = \cox(\Gamma)$.  If the root system
  $\Phi_\Gamma \cap \dir(L)^\perp$ of horizontal roots is reducible,
  then the interval $[1,w]$ contains a bowtie, it is not a lattice and
  the dual Artin group $\dart(\Gamma,w)$ is not a Garside group.
\end{thm}

\begin{proof}
  Let $\sigma$ be a chamber where the reflections defined by its
  facets can be multiplied in an appropriate order to produce $w = r_0
  r_1 \cdots r_n$.  By repeatedly applying Lemma~\ref{lem:sink-source}
  and replacing $\sigma$ with another chamber if necessary, we may
  assume without loss of generality that the reflection corresponding
  to the white dot is the reflection $r_0$.  When this is the case,
  the remaining reflections in the factorization fix a vertex $x$ of
  $\sigma$ with the property that every hyperplane of a reflection in
  $W$ is parallel to a hyperplane of a reflection in $W$ fixing $x$.
  In other words, we can identify $x$ as our origin as in
  Definition~\ref{def:root-cox}, the group generated by $r_1$ through
  $r_n$ is the group $W_0$, and the product $r_1 r_2 \cdots r_n$ is a
  Coxeter element for $W_0$.  Because every reflection through $x$
  occurs in some minimal length factorization of this Coxeter element,
  we may modify the product $r_1r_2 \cdots r_n$ so that $r_1$ is a
  reflection with a hyperplane parallel to the hyperplane of $r_0$.
  As a consequence $r_0r_1$ is a translation in a root direction and
  concretely a translation in the direction $\lambda$ that is the root
  of highest weight relative to this simple system.

  Next, consider the element $r_2 r_3 \cdots r_n$.  This is an
  elliptic element fixing a line $L'$ and since it differs from $w$ by
  a translation, $L'$ is parallel to $L$, the Coxeter axis of $w$.
  Moreover, $L'$ must lie in the hyperplane fixed by $r_i$ for each
  $i\geq 2$, so the roots of these reflections belong to the root
  system $\Phi_{\textrm{hor}} = \Phi_\Gamma \cap \dir(L)^\perp$.  We
  should note that the common root $\lambda$ of $r_0$ and $r_1$ is not
  in this root system because they are now the only roots capable of
  moving points in a direction that includes motion in the direction
  of $L$.

  By hypothesis this system of horizontal roots is reducible, say
  $\Phi_{\textrm{hor}} = \Phi_1 \cup \cdots \cup \Phi_j$ for some $j
  >1$ with each $\Phi_i$ an irreducible root system that spans a
  subspace $V_i$ and $V_1 \oplus \cdots \oplus V_j$ is an orthogonal
  decomposition of $\dir(L)^\perp$.  Because the reflections $r_i$, $2
  \leq i\leq n$ form a minimal length factorization of an elliptic
  isometry, their roots are linearly independent
  \cite[Lemma~$\LemMinEll$]{BrMc-factor}.  Thus the number of roots in
  each $\Phi_i$ is bounded above by the dimension of the corresponding
  $V_i$.  Moreover, since the number of reflections equals the
  dimension of $\dir(L)^\perp$, the roots of these reflections that
  lie in each $\Phi_i$ form a basis for $V_i$.

  Finally, note that because the reflections $r_0, r_1, \ldots, r_n$
  generate all of $W$, their roots must generate the entire
  irreducible root system $\Phi_\Gamma$.  In partiuclar, the vector
  $\lambda$ is not in $V_i^\perp$ for any $i$ since this would lead to
  an obvious decomposition.  More to the point, there must be an $r_i$
  with a root in $\Phi_1$ is not orthogonal to $\lambda$ since these
  roots form a basis of $V_1$, and the same is true for the other
  components as well.  In particular, we can select two reflections
  from our given factorization of $w$ with roots from distinct
  irreducible components of the horizontal root system $\Phi_\Gamma
  \cap \dir(L)^\perp$ and that neither one commutes with the
  translation $t = r_0r_1$ and by Lemma~\ref{lem:rewriting} we may
  assume that they are $r_2$ and $r_3$.  At this point, we can apply
  Proposition~\ref{prop:bowtie} with $t=r_0 r_1$, $r=r_2$ and $r' =
  r_3$ to conclude that the interval from $1$ to $r_0r_1r_2r_3$
  contains a bowtie $(a,b:c,d)$.  But this interval is contained
  inside $[1,w]$ and the additional elements cannot resolve the bowtie
  since any element below $a$ and $b$ was already contained in
  $[1,r_0r_1r_2r_3]$.  Thus $[1,w]$ contains a bowtie and is not a
  lattice.
\end{proof}

Consider the case of the $\wt G_2$ Coxeter group.  By comparing
Figures~\ref{fig:g2-axis} and~\ref{fig:phi-g2} it is clear that there
are exactly $2$ horizontal roots forming a $\Phi_{A_1}$ root system.
Since this is irreducible, Theorem~\ref{thm:bowtie} does not apply.
This is consistent with Theorem~\ref{thm:g2} where we showed that the
interval used to define the dual presentation of $\art(\wt G_2)$ is a
lattice and thus the dual presentation is a Garside presentation.

%%%%%%%%%%%%%%%%%%%%%%%%%%%%%%%%%%%%%%%%%%%%%%%%%%%%%%%%%
\section{Computations and remarks}\label{sec:computation}
%%%%%%%%%%%%%%%%%%%%%%%%%%%%%%%%%%%%%%%%%%%%%%%%%%%%%%%%%

In this final section we use Theorem~\ref{thm:bowtie} to complete the
proof of our main theorem, Theorem~\ref{main:dual}.  At this point it
is relatively straightforward.  For each type and representative
Coxeter element we compute the direction of the Coxeter axis using
Remark~\ref{rem:axis-dir} and then we compute the system of horizontal
roots.  In each case not covered by \cite{Digne06}, \cite{Digne12}, or
Theorem~\ref{thm:g2}, the system of horizontal roots is reducible.  We
begin by introducing the relevant root systems using a slightly
idiosyncratic notation that John Crisp and I have found to be quite
useful when performing explicit root computations by hand.

\begin{defn}[Root notation]\label{def:root-not}
  In almost every standard root system for a euclidean Coxeter group,
  a root is completely specified by indicating the location and the
  sign of its nonzero entries.  This is because all nonzero entries
  has the same absolute value and this common value only depends on
  the number of nonzero entries.  We list the locations of the nonzero
  entries in the subscript together with a slash ``/''.  The locations
  of the positive entries occur before the slash and the locations of
  the negative entries occur afterwards.  For example, $r_{ij/}$,
  $r_{i/j}$ and $r_{/ij}$ denote the vectors $e_i+e_j$, $e_i - e_j$
  and $-e_i - e_j$, respectively and in the $E_8$ root system the
  vector $\frac12 (1,-1,1,-1,1,1,-1,-1)$ is written $r_{1356/2478}$.
\end{defn}

\begin{defn}[Root systems]\label{def:root-sys}
  Let $\Phi_k^{(n)}$ be the collection of $2^k \binom{n}{k}$ vectors
  of the form: \[\Phi_k^{(n)} = \{\pm e_{i_1} \pm e_{i_2} \pm \cdots
  \pm e_{i_k} \mid 1 \leq i_1 < i_2 < \cdots i_k \leq n \}\] and let
  $\Phi_{k,\textrm{even}}^{(n)}$ be the subset of these vectors with
  an even number of minus signs.  The standard root systems of types
  $B_n$, $C_n$, $D_n$, $F_4$ and $E_8$ are very easy to describe using
  this notation.  See Table~\ref{tbl:easy-roots}.  The others involve
  slight modifications.  The roots in $\Phi_{C_n}$ that are orthogonal
  to the vector $(1^n)$, i.e. the vector with all $n$ coordinates
  equal to $1$, form the standard $A_{n-1}$ root system.  The roots in
  $\Phi_{E_8}$ orthogonal to $r_{7/8}$, i.e. the roots with $x_7 =
  x_8$, form the standard $E_7$ root system.  And the roots in
  $\Phi_{E_8}$ orthogonal to $r_{6/7}$ and $r_{7/8}$, i.e. the roots
  with $x_6 = x_7 = x_8$, form the standard $E_6$ root system.
\end{defn}

\begin{table}
  \begin{tabular}{|c|l|} \hline
    Type & Roots\\ \hline \hline
    $B_n$ & $\Phi_2^{(n)} \cup \Phi_1^{(n)}$\\ \hline
    $C_n$ & $\Phi_2^{(n)} \cup 2\Phi_1^{(n)}$\\ \hline
    $D_n$ & $\Phi_2^{(n)}$\\ \hline
    $F_4$ & $\Phi_2^{(4)} \cup 2\Phi_1^{(4)} \cup \Phi_4^{(4)}$\\ \hline
    $E_8$ & $\Phi_2^{(8)} \cup \frac12 \Phi_{8,\textrm{even}}^{(8)}$\\ \hline
  \end{tabular} \vspace*{1em}
  \caption{Roots systems with simple descriptions.\label{tbl:easy-roots}}
\end{table}

We begin with the four classical families.  The computations for type
$C$ are done in greater detail because they are straightforward and
indicate the kinds of computations that are merely sketched for the
remaining types.

\begin{exmp}[Type $C$]
  We begin by selecting vectors from the $\Phi_{C_n}$ root system
  whose nonpositive dot products are encoded in the $\wt C_n$ diagram
  of Figure~\ref{fig:dynkin}.  From left to right we use $2e_1$,
  $-e_1-e_2$, $e_2+e_3$, \ldots $(-1)^{n-1}(e_{n-1}+e_n)$ and $(-1)^n
  2e_n$.  In the notation of Definition~\ref{def:root-not} these are
  the vectors $r_{1/}$, $(-1)^i r_{i(i+1)/}$ and $(-1)^n r_{n/}$.  Our
  choice of vectors is nonstandard, but it has the advantage of
  producing a bipartite Coxeter element whose axis is in a direction
  that makes computing and identifying the horizontal root system
  trivial.  We compute the direction of the Coxeter axis following the
  procedure outlined in Remark~\ref{rem:axis-dir}.  The unique linear
  dependency among these vectors involves adding the first and last
  vectors to two times each of the remaining vectors.  The bipartite
  structure separates them based on parity in the list and the sum of
  the odd terms is the vector $(2,2,\ldots, 2)$, or $(2^n)$ in
  Conway's shorthand notation.  The roots orthogonal to this direction
  are those of the form $r_{i/j} = e_i - e_j$ and these form the
  irreducile root system $\Phi_{A_{n-1}}$.  Note that
  Theorem~\ref{thm:bowtie} is not applicable and that this is
  consistent with the results in \cite{Digne12} where Digne
  established that the interval that defines the dual euclidean Artin
  group of type $\widetilde C_n$ is, in fact, a lattice and the dual
  presentation of $\art(\wt C_n)$ is a Garside presentation.
\end{exmp}

Type $B$ is very similar but has a reducible horizontal root system.

\begin{exmp}[Type $B$]
  Next consider the root system of type $B_n$ and select the vectors
  $e_1$, $-e_1-e_2$, $e_2+e_3$, \ldots, $(-1)^{n-2}(e_{n-2}+
  e_{n-1})$, $(-1)^{n-1}(e_{n-1}+ e_n)$ and $(-1)^{n-1}(e_{n-1} -
  e_n)$ to represent the vertices of the $\wt B_n$ diagram of
  Figure~\ref{fig:dynkin} from left to right.  In shorthand notation,
  these are the vectors $r_{1/}$, $r_{/12}$, $r_{23/}$, \ldots
  $(-1)^{n-2} r_{(n-2)(n-1)/}$, $(-1)^{n-1} r_{(n-1)n/}$ and
  $(-1)^{n-1} r_{(n-1)/n}$, keeping in mind that $r_{1/}$ denotes the
  vector $e_1$ in the $B_n$ root system and not $2e_1$ as in the $C_n$
  root system.  The unique linear dependency among these vectors is
  obtained by adding the two final vectors to two times each of the
  remaining vectors.  In $\wt B_4$, for example, the linear dependency
  is $2 r_{1/} + 2 r_{/12} + 2 r_{23/} + r_{/34} + r_{4/3} =
  (0,0,0,0)$.  Separating the terms based on the bipartite structure
  shows that the direction of the Coxeter axis is $(2,2,\ldots,2,0) =
  (2^{n-1} 0)$.  The horizontal roots are those of the form $r_{i/j} =
  e_i -e_j$ with $i,j < n$ and the two roots $\pm r_{n/} = \pm e_n$.
  This is clearly a reducible horizontal root system with the
  components isomorphic to $\Phi_{A_{n-2}}$ and $\Phi_{A_1}$.  In
  particular, Theorem~\ref{thm:bowtie} applies, the interval used to
  define the dual euclidean Artin group of type $\wt B_n$ is not a
  lattice and the dual presentation of $\art(\wt B_n)$ is not a
  Garside presentation.
\end{exmp}

Type $D$ is another slight variation.

\begin{exmp}[Type $D$]
  Consider the root system of type $D_n$ and select the vectors
  $e_1-e_2$, $-e_1-e_2$, $e_2+e_3$, \ldots, $(-1)^{n-2}(e_{n-2}+
  e_{n-1})$, $(-1)^{n-1}(e_{n-1}+ e_n)$ and $(-1)^{n-1}(e_{n-1} -
  e_n)$ to represent the vertices of the $\wt D_n$ diagram of
  Figure~\ref{fig:dynkin} from left to right.  The unique linear
  dependency among these vectors is obtained by adding the first two
  vectors, the last two vectors, and two times each of the remaining
  vectors.  In $\wt D_5$, for example, the linear dependency is
  $r_{1/2} + r_{/12} + 2 r_{23/} + 2 r_{/34} + r_{45/} + r_{4/5} =
  (0,0,0,0,0)$.  Separating the terms based on the bipartite structure
  shows that the direction of the Coxeter axis is $(0,2,2,\ldots,2,0)
  = (0 2^{n-2} 0)$.  The horizontal roots are those of the form
  $r_{i/j} = e_i -e_j$ with $1< i,j < n$ and the four roots $\pm
  r_{1n/}$ and $\pm r_{1/n}$ This is a reducible horizontal root
  system with three irreducible factors (since $r_{1n/}$ and $r_{1/n}$
  are orthogonal) that are isomorphic to $\Phi_{A_{n-3}}$,
  $\Phi_{A_1}$ and $\Phi_{A_1}$.  In particular,
  Theorem~\ref{thm:bowtie} applies, the interval used to define the
  dual euclidean Artin group of type $\wt D_n$ is not a lattice and
  the dual presentation of $\art(\wt D_n)$ is not a Garside
  presentation.
\end{exmp}

The final classical family is type $A$.

\begin{exmp}[Type $A$]
  Let $W$ be the Coxeter group $\cox(\wt A_n)$.  For each $(p,q)$ with
  $p+q = n+1$ and $p \geq q \geq 1$ we can construct a $(p,q)$-bigon
  Coxeter element $w$ as follows.  First let $W$ act on $\R^{n+1}$ in
  the natural way, permuting coordinates and translating along vectors
  orthogonal to $(1^{n+1})$.  Next label these $n+1$ coordinates
  $(x_1,x_2 \ldots, x_p, y_1, y_2, \ldots, y_q)$.  Let the unique
  source in the acyclic orientation be the reflection that swaps
  coordinates $x_1$ and $y_1$, let the vertices along one side of the
  bigon represent reflections that swap $x_i$ and $x_{i+1}$ in
  ascending order and let the vertices along the other side represent
  reflections that swap $y_i$ and $y_{i+1}$ in ascending order.
  Finally, let the unique sink represent the reflection that sends the
  coordinates $(x_p,y_q)$ to $(y_q-1,x_p+1)$.  This is a reflection
  fixing the hyperplane $y_q = x_p+1$.  The product of these
  reflections in this order is an isometry of $\R^{n+1}$ that sends
  $(x_1,x_2,\ldots, x_p, y_1, y_2, \ldots, y_q)$ to $(x_p + 1, x_1,
  x_2, \ldots, x_{p-1}, y_q-1, y_1, y_2, \ldots, y_{q-1})$.  Although
  we cannot use the method of Remark~\ref{rem:axis-dir} to calculate
  the direction of the Coxeter axis, it is easy enough to compute that
  its $pq$-th power of this motion is a pure translation in the
  direction $(p^q,-q^p)$, i.e. the vector with its first $q$
  coordinates equal to $p$ and its next $p$ coordinates equal to $-q$.
  Thus this is the direction of the axis of this Coxeter element.

  As a consequence, the horizontal roots are those of the form
  $r_{i/j} = e_i -e_j$ with $i,j \leq p$ or with $i,j > p$.  So long
  as $q$ (and therefore $p$) is at least $2$, there is at least one
  horizontal root of each type and the horizontal root system is
  reducible with one component isomorphic to a $\Phi_{A_{p-1}}$ and
  the other component isomorphic to a $\Phi_{A_{q-1}}$.  In
  particular, Theorem~\ref{thm:bowtie} applies to any $(p,q)$-bigon
  Coxeter element with $p\geq q \geq 2$, the interval used to define
  the dual euclidean Artin group of type $\wt A_n$ with respect to
  this Coxeter element is not a lattice and the corresponding dual
  presentation of $\art(\wt A_n)$ is not a Garside presentation.  In
  the remaining case where $q=1$ and $p=n$, the horizontal roots for
  an irreducible $\Phi_{A_{n-1}}$ root system and
  Theorem~\ref{thm:bowtie} is not applicable.  This is consistent with
  the results in \cite{Digne06} where Digne established that the
  interval that defines the dual euclidean Artin group of type
  $\widetilde A_n$ with respect to an $(n,1)$-bigon Coxeter element
  is, in fact, a lattice and the corresponding dual presentation is a
  Garside presentation.
\end{exmp}

And finally we shift our attention to the exceptional types.  Since
type $G$ is covered by Theorem~\ref{thm:g2}, we only need to discuss
the four remaining examples.  In each case we list the vectors chosen
to as our simple system, the resulting direction of the Coxeter axis
and the vectors of the horizontal root system grouped into irreducible
components.  These computations were initially carried out by hand and
then a few lines of code in \texttt{GAP} were used to doublecheck and
validate these results.

\begin{exmp}[Type $F$]
  Consider the root system of type $F_4$.  If we select the vectors
  $r_{1/2}$, $r_{23/}$, $r_{/1234}$, $r_{4/}$, and $r_{12/34}$ from
  the $\Phi_{F_4}$ root system as the vectors represented in the
  extended Dynkin diagram of type $\wt F_4$, then the direction of the
  axis of the corresponding bipartite Coxeter element is $(0,1,1,2)$.
  There are $8$ roots that are horizontal with respect to this axis
  and they split into two irreducible factors.  There is a
  $\Phi_{A_1}$ root system formed by the roots $\pm \{ r_{2/3} \}$ and
  a $\Phi_{A_2}$ root system formed by the roots $\pm \{ r_1,
  r_{123/4} , r_{23/14} \}$.  As a consequence
  Theorem~\ref{thm:bowtie} applies, the interval used to define the
  dual euclidean Artin group of type $\wt F_4$ is not a lattice and
  the dual presentation of $\art(\wt F_4)$ is not a Garside
  presentation.
\end{exmp}

\begin{exmp}[Type $E$]
  Consider the root system of type $E_6$.  If we select the vectors
  $r_{12/}$, $r_{5/2}$, $r_{/45}$, $r_{4/3}$, $r_{235678/14}$,
  $r_{2345/1678}$ and $r_{134678/25}$ from the $\Phi_{E_6}$ root
  system as the vectors represented in the extended Dynkin diagram of
  type $\wt E_6$, then the direction of the axis of the corresponding
  bipartite Coxeter element is $(1,1,1,-3,-3,1,1,1)$.  There are $14$
  roots that are horizontal with respect to this axis and they split
  into three irreducible factors.  There is a $\Phi_{A_2}$ root system
  formed by the roots $\pm \{ r_{1/2}, r_{2/3}, r_{1/3} \}$, another
  $\Phi_{A_2}$ root system formed by the roots $\pm \{ r_{4/5},
  r_{1234/5678}, r_{1235/4678} \}$, and a $\Phi_{A_1}$ root system
  formed by the roots $\pm \{ r_{12345678/} \}$.

  Next consider the root system of type $E_7$.  If we select the
  vectors $r_{/15}$, $r_{12/}$, $r_{/27}$, $r_{78/}$, $r_{/38}$,
  $r_{34/}$, $r_{/46}$ and $r_{2356/1478}$ from the $\Phi_{E_7}$ root
  system as the vectors represented in the extended Dynkin diagram of
  type $\wt E_7$, then the direction of the axis of the corresponding
  bipartite Coxeter element is $(1,1,1,1,0,0,2,2)$.  There are $20$
  roots that are horizontal with respect to this axis and they split
  into three irreducible factors.  There is a $\Phi_{A_3}$ root system
  formed by the roots $\pm \{ r_{1/2}, r_{1/3}, r_{1/4}, r_{2/3},
  r_{2/4}, r_{3/4} \}$, a $\Phi_{A_2}$ root system formed by the roots
  $\pm \{ r_{56/}, r_{1234/5678}, r_{12345/78} \}$, and a $\Phi_{A_1}$
  root system formed by the roots $\pm \{ r_{5/6} \}$.

  Finally consider the root system of type $E_8$.  If we select the
  vectors $r_{12/}$, $r_{/25}$, $r_{5/6}$, $r_{6/7}$, $r_{78/}$,
  $r_{/38}$, $r_{34/}$, $r_{28/134567}$ and $r_{2367/1458}$ from the
  $\Phi_{E_8}$ root system as the vectors represented in the extended
  Dynkin diagram of type $\wt E_8$, then the direction of the axis of
  the corresponding bipartite Coxeter element is $(1,1,1,1,3,-3,2,2)$.
  There are $28$ roots that are horizontal with respect to this axis
  and they split into three irreducible factors.  There is a
  $\Phi_{A_4}$ root system formed by the roots $r_{1/2}$, $r_{1/3}$,
  $r_{1/4}$, $r_{2/3}$, $r_{2/4}$, $r_{3/4}$, $r_{15/234678}$,
  $r_{25/134678}$, $r_{35/124678}$, $r_{45/123678}$ and their
  negatives, a $\Phi_{A_2}$ root system formed by the roots $\pm \{
  r_{56/}, r_{1234/5678}, r_{123456/78}\}$ and a $\Phi_{A_1}$ root
  system formed by the roots $\pm\{r_{7/8}\}$.

  In each case, Theorem~\ref{thm:bowtie} applies, the interval used to
  define the dual euclidean Artin group of type $\wt E_n$ for $n = 6$,
  $7$ or $8$ is not a lattice and the dual presentation of $\art(\wt
  E_n)$ is not a Garside presentation.
\end{exmp}

At this point we have shown that for each type and for each geometric
equivalence class of Coxeter elements not covered by earlier results,
the resulting interval is not a lattice and the corresponding dual
presentation is not a Garside presentation.  This completes the proof
of Theorem~\ref{main:dual}.  As a final comment we note that the
existence of a uniform reason for the failure of the lattice property
in all of these cases (i.e. reducibility of the horizontal root
system) leads one to hope for the existence of a uniform way to work
around the problem.  This indeed turns out to be the case as Robert
Sulway and I show in \cite{McSu-artin-euclid} where we clarifying the
basic structural properties of all euclidean Artin groups.

\newcommand{\etalchar}[1]{$^{#1}$}
\def\cprime{$'$}
\providecommand{\bysame}{\leavevmode\hbox to3em{\hrulefill}\thinspace}
\providecommand{\MR}{\relax\ifhmode\unskip\space\fi MR }
\providecommand{\MRhref}[2]{%
  \href{http://www.ams.org/mathscinet-getitem?mr=#1}{#2}
}
\providecommand{\href}[2]{#2}

%%%%%%%%%%%%%%

\begin{thebibliography}{CMW04}

\bibitem[BDSW]{BDSW-hurwitz}
Barbara Baumeister, Matthew Dyer, Christian Stump, and Patrick Wegener, \emph{A
  note on the transitive {H}urwitz action on decompositions of parabolic
  {C}oxeter elements}, arXiv:1402.2500 [math.GR].

\bibitem[Bes]{Bessis-top}
David Bessis, \emph{Topology of complex reflection arrangements}, Available at
  {\tt arXiv:math.GT/0411645}.

\bibitem[Bes03]{Be03}
\bysame, \emph{The dual braid monoid}, Ann. Sci. \'Ecole Norm. Sup. (4)
  \textbf{36} (2003), no.~5, 647--683. \MR{MR2032983 (2004m:20071)}

\bibitem[BH99]{BridsonHaefliger99}
Martin~R. Bridson and Andr{\'e} Haefliger, \emph{Metric spaces of non-positive
  curvature}, Grundlehren der Mathematischen Wissenschaften [Fundamental
  Principles of Mathematical Sciences], vol. 319, Springer-Verlag, Berlin,
  1999. \MR{2000k:53038}

\bibitem[BM]{BrMc-factor}
Noel Brady and Jon McCammond, \emph{Factoring euclidean isometries}, To appear
  in the {\it International Journal of Algebra and Computation} {\tt
  arXiv:1312.7780}.

\bibitem[BM00]{BrMc00}
Thomas Brady and Jonathan~P. McCammond, \emph{Three-generator {A}rtin groups of
  large type are biautomatic}, J. Pure Appl. Algebra \textbf{151} (2000),
  no.~1, 1--9. \MR{2001f:20076}

\bibitem[BM10]{BrMc10}
Tom Brady and Jon McCammond, \emph{Braids, posets and orthoschemes}, Algebr.
  Geom. Topol. \textbf{10} (2010), no.~4, 2277--2314.

\bibitem[CB92]{Crawley-Boevey92}
William Crawley-Boevey, \emph{Exceptional sequences of representations of
  quivers}, Proceedings of the {S}ixth {I}nternational {C}onference on
  {R}epresentations of {A}lgebras ({O}ttawa, {ON}, 1992) (Ottawa, ON),
  Carleton-Ottawa Math. Lecture Note Ser., vol.~14, Carleton Univ., 1992, p.~7.
  \MR{1206935 (94c:16017)}

\bibitem[CMW04]{ChMeWh04}
R.~Charney, J.~Meier, and K.~Whittlesey, \emph{Bestvina's normal form complex
  and the homology of {G}arside groups}, Geom. Dedicata \textbf{105} (2004),
  171--188. \MR{MR2057250 (2005e:20057)}

\bibitem[DDG{\etalchar{+}}]{DDGKM-garside}
Patrick Dehornoy, Fran\c{c}ois Digne, Eddy Godelle, Daan Krammer, and Jean
  Michel, \emph{Foundations of garside theory}, {\tt arXiv:1309.0796}.

\bibitem[DDM13]{DehornoyDigneMichel13}
Patrick Dehornoy, Fran{\c{c}}ois Digne, and Jean Michel, \emph{Garside families
  and {G}arside germs}, J. Algebra \textbf{380} (2013), 109--145. \MR{3023229}

\bibitem[Dig06]{Digne06}
F.~Digne, \emph{Pr\'esentations duales des groupes de tresses de type affine
  {$\widetilde A$}}, Comment. Math. Helv. \textbf{81} (2006), no.~1, 23--47.
  \MR{2208796 (2006k:20075)}

\bibitem[Dig12]{Digne12}
\bysame, \emph{A {G}arside presentation for {A}rtin-{T}its groups of type
  {$\widetilde C_n$}}, Ann. Inst. Fourier (Grenoble) \textbf{62} (2012), no.~2,
  641--666. \MR{2985512}

\bibitem[DP99]{DePa99}
Patrick Dehornoy and Luis Paris, \emph{Gaussian groups and {G}arside groups,
  two generalisations of {A}rtin groups}, Proc. London Math. Soc. (3)
  \textbf{79} (1999), no.~3, 569--604. \MR{2001f:20061}

\bibitem[DP02]{DaPr02}
B.~A. Davey and H.~A. Priestley, \emph{Introduction to lattices and order},
  second ed., Cambridge University Press, New York, 2002. \MR{MR1902334
  (2003e:06001)}

\bibitem[Hum90]{Humphreys90}
James~E. Humphreys, \emph{Reflection groups and {C}oxeter groups}, Cambridge
  Studies in Advanced Mathematics, vol.~29, Cambridge University Press,
  Cambridge, 1990. \MR{1066460 (92h:20002)}

\bibitem[Igu11]{Igusa11}
Kiyoshi Igusa, \emph{Exceptional sequences, braid groups and clusters}, Groups,
  algebras and applications, Contemp. Math., vol. 537, Amer. Math. Soc.,
  Providence, RI, 2011, pp.~227--240. \MR{2799103 (2012f:16001)}

\bibitem[IS10]{IgusaSchiffler10}
Kiyoshi Igusa and Ralf Schiffler, \emph{Exceptional sequences and clusters}, J.
  Algebra \textbf{323} (2010), no.~8, 2183--2202. \MR{2596373 (2011b:20118)}

\bibitem[IT09]{IngallsThomas09}
Colin Ingalls and Hugh Thomas, \emph{Noncrossing partitions and representations
  of quivers}, Compos. Math. \textbf{145} (2009), no.~6, 1533--1562.
  \MR{2575093 (2010m:16021)}

\bibitem[McCa]{McCammond-cont-braids}
Jon McCammond, \emph{Pulling apart orthogonal groups to find continuous
  braids}, Preprint 2010.

\bibitem[McCb]{Mc-artin-survey}
\bysame, \emph{The structure of euclidean {A}rtin groups}, To appear in the
  Proceedings of the 2013 Durham conference on Geometric and Cohomological
  group theory {\tt arXiv:1312.7781}.

\bibitem[MP11]{McCammondPetersen-refl-bound}
Jon McCammond and T.~Kyle Petersen, \emph{Bounding reflection length in an
  affine {C}oxeter group}, J. Algebraic Combin. \textbf{34} (2011), no.~4,
  711--719. \MR{2842917 (2012h:20089)}

\bibitem[MS]{McSu-artin-euclid}
Jon McCammond and Robert Sulway, \emph{Artin groups of euclidean type}, {\tt
  arXiv:1312.7770}.

\bibitem[Rin94]{Ringel94}
Claus~Michael Ringel, \emph{The braid group action on the set of exceptional
  sequences of a hereditary {A}rtin algebra}, Abelian group theory and related
  topics ({O}berwolfach, 1993), Contemp. Math., vol. 171, Amer. Math. Soc.,
  Providence, RI, 1994, pp.~339--352. \MR{1293154 (95m:16006)}

\bibitem[Squ87]{Squier87}
Craig~C. Squier, \emph{On certain {$3$}-generator {A}rtin groups}, Trans. Amer.
  Math. Soc. \textbf{302} (1987), no.~1, 117--124. \MR{887500 (88g:20069)}

\bibitem[ST89]{SnapperTroyer89}
Ernst Snapper and Robert~J. Troyer, \emph{Metric affine geometry}, second ed.,
  Dover Books on Advanced Mathematics, Dover Publications Inc., New York, 1989.
  \MR{1034484 (90j:51001)}

\bibitem[Sta97]{EC1}
Richard~P. Stanley, \emph{Enumerative combinatorics. {V}ol. 1}, Cambridge
  Studies in Advanced Mathematics, vol.~49, Cambridge University Press,
  Cambridge, 1997, With a foreword by Gian-Carlo Rota, Corrected reprint of the
  1986 original. \MR{98a:05001}

\end{thebibliography}
\end{document}